\documentclass[11pt]{article}
\usepackage{amsmath}
\usepackage{amsfonts}
\usepackage{amsmath,amsthm}
\usepackage{amsmath,amssymb,amsthm,latexsym}
\usepackage{amscd}

\usepackage{color}
\usepackage[all]{xy}
\usepackage{hyperref}
\numberwithin{equation}{section}
\newtheorem{theorem}{Theorem}[section]
\newtheorem{proposition}[theorem]{Proposition}

\newtheorem{lemma}[theorem]{Lemma}

\theoremstyle{remark}
\newtheorem{remark}[theorem]{Remark}

\def\vec{\mathbf}
\def\<{\langle}
\def\>{\rangle}
\def\L{\mathcal{L}}

\def\R{\mathbb{R}}
\def\C{\mathbb{C}}

\def\Re{{\rm Re}}
\def\Im{{\rm Im}}
\def\zb{\bar z}
\def\d{{\rm d}}
\newcommand{\abs}[1]{|#1|}

\begin{document}
\title{\bf{Classification of Willmore 2-spheres in the 5-dimensional sphere}}

\author{Xiang Ma%
  \thanks{LMAM, School of Mathematical Sciences, Peking University,
 Beijing 100871, China. \texttt{maxiang@math.pku.edu.cn}}
\and Changping Wang \thanks{College of Mathematics and Computer Sciences, Fujian Normal University, Fuzhou 350108, China. \texttt{cpwang@fjnu.edu.cn}}
\and Peng Wang%
  \thanks{Department of Mathematics, Tongji University, Shanghai 200092, China. \texttt{netwangpeng@tongji.edu.cn}, Corresponding author.}}

\date{\today}

\maketitle

\begin{center}
{\bf Abstract}
\end{center}

The classification of Willmore 2-spheres in the $n$-dimensional sphere $S^n$ is a long-standing problem, solved only when $n=3,4$ by Bryant, Ejiri, Musso and Montiel independently. In this paper we give a classification when $n=5$. There are three types of such surfaces up to M\"obius transformations: (1) super-conformal surfaces in $S^4$; (2) minimal surfaces in $R^5$; (3) adjoint transforms of super-conformal minimal surfaces in $R^5$. In particular, Willmore surfaces in the third class are not S-Willmore (i.e., without a dual Willmore surface).

To show the existence of Willmore 2-spheres in $S^5$ of type (3), we describe all adjoint transforms of a super-conformal minimal surfaces in $R^n$ and provide some explicit criterions on the immersion property. As an application, we obtain new immersed Willmore 2-spheres in $S^5$ and $S^6$, which are not S-Willmore.

\hspace{2mm}

{\bf Keywords:}
Willmore surfaces,
adjoint transform,
minimal surfaces,
super-conformal surfaces,
pedal surfaces,
harmonic maps\\

{\bf MSC(2000):\hspace{2mm} 53A10, 53C42, 53C45}

\section{Introduction}

Willmore surfaces immersed in a real space form $M^n(c)$ of constant sectional curvature $c$
are critical surfaces with respect to the Willmore functional
\[\int (|H|^2-K+c)\d A,\]
where $H$ is the mean curvature vector, $K$ is the Gauss curvature,
and $\d A$ is the area element with respect to the induced metric.

It is well-known that minimal surfaces in $M^n(c)$ are special examples of Willmore surfaces. Moreover, the Willmore functional is a conformal invariant, which implies that Willmore surfaces form a conformally invariant surface class. Thus Willmore surfaces are natural generalizations of minimal surfaces in M\"obius geometry. In particular, Willmore surfaces can always be regarded as located in $S^n$ and the classification is generally up to the action of the M\"obius group on $S^n$.

Under the global assumption that the surface is a topological 2-sphere $S^2$, one can utilize the vanishing theorem of holomorphic forms on $S^2$ to deduce many beautiful results, which has been demonstrated in the classical work by Hopf and Calabi. Comparing to the theory on minimal 2-spheres in $S^n$ and in other symmetric spaces (see the seminal work of Calabi \cite{calabi}, Chern \cite{Chern}, Uhlenbeck \cite{Uh}, Burstall and Guest \cite{burstall}, Burstall and Rawnsley \cite{burstall3}), it is more difficult to classify Willmore 2-spheres in $S^n$. Concerning the cases $n\le 4$, a beautiful description has been derived as below.

\begin{theorem}~\cite{bryant1, Ejiri, Montiel, Musso}~~ A Willmore 2-sphere in $\mathbb{S}^4$ belongs to one of the following two surface classes (up to a M\"obius transformation):
\begin{enumerate}
\item  Complete minimal surfaces in $\mathbb{R}^4$ with embedded flat ends.
\item The twistor projection image of rational curves in $\mathbb{C}P^3$.
\end{enumerate}
Moreover, the Willmore functional of them are all integer multiples of $4\pi$ (\cite{bryant1, Montiel}). In particular, in $S^3$ we have only examples in the first class.
 \end{theorem}

For the case $n>4$, there has been no much progress on this problem after 2000 \cite{Montiel} except that the first author introduced in 2005 the so-called \emph{adjoint transforms} of \emph{immersed} Willmore surfaces \cite{ma0}. The thesis \cite{ma0} suggested a procedure to reduce Willmore 2-spheres to Euclidean minimal surfaces by repeatedly applying some canonically chosen adjoint transforms. The main difficulty is that  generally  such adjoint transforms will produce branch points. If such branch points appear, one need to deal with possible poles in the construction of holomorphic forms. Even worse, it seems doubtful whether one can continue this sequence of adjoint transforms around such singularities.

In this paper, we overcome this difficulty for the case $n=5$ and derive a classification of all Willmore $S^2$ in $S^5$. The main theorem is stated as below.

\begin{theorem}\label{main}
A Willmore 2-sphere $y:S^2\to S^5$ is M\"obius equivalent to one surface in either of the following three classes:
\begin{enumerate}
\item Super-conformal surfaces in $S^4$ (coming from the twistor curves in $\mathbb{C}P^3$);
\item Minimal surfaces in $\mathbb{R}^5$ with genus $0$ and embedded flat ends;
\item Adjoint transforms of complete isotropic minimal surfaces of genus $0$ in $\mathbb{R}^5$.
\end{enumerate}
Moreover, different from the first two cases, surfaces of the third class are not S-Willmore.
\end{theorem}

To have a better understanding about this result, recall that
an central theme in the study of Willmore surfaces is to determine the infimum of the Willmore functional among various immersions of the 2-sphere $S^2$ and the torus $T^2$.
When the genus $g$ is arbitrary, Simon \cite{Simon} and Kuwert  etc. \cite{Kuwert,Ku-Li-Sc} has shown that the infimum is always attained by some Willmore surfaces.  In \cite{Kusner1989} such surfaces are conjectured to be Lawson's minimal surfaces \cite{Lawson}. Along this way, the Willmore conjecture was finally proved by Marques and Neves in 2012 (see \cite{Marques}).

More generally, for any fixed integer $g$, it is desirable to give a general construction and classification of closed oriented Willmore surfaces of genus $g$, and to determine the values of their Willmore functional (the possible critical values).
Basic examples include complete Euclidean minimal surfaces with embedded flat ends (compactified at the infinity of $\mathbb{R}^n$, see \cite{bryant1,Bryant1988, kusner,PX}) and closed minimal surfaces in $\mathbb{S}^n$ \cite{Lawson}, \cite{Kusner1989}. But there are much more examples other than minimal ones in space forms, like the Willmore Hopf tori \cite{Pinkall}.

The genus zero case is the simplest case.
From a retrospective viewpoint, the successful classification of Willmore 2-spheres in $\mathbb{S}^3$ by Bryant \cite{bryant1} is based on the following crucial observations:
\begin{enumerate}
\item (Conformal Gauss map.) There is a so-called \emph{conformal Gauss map} into the de-Sitter space $\mathbb{S}^4_1$, which induces a M\"obius invariant metric for a surface in $\mathbb{S}^3$. Geometrically speaking, this map corresponds to the \emph{mean curvature sphere}, which is tangent to the original surface at one point with the same mean curvature (vector).
\item (Harmonic map.) An immersed surface $M$ in $\mathbb{S}^3$ is Willmore if, and only if, this conformal Gauss map is a harmonic map. This reveals the connection between Willmore surfaces and integrable systems.
\item (Duality Theorem.) Any Willmore surface $f:M\to S^3$ allows a dual Willmore surface $\hat{f}:M\to S^3$ arising as the second envelope of the mean curvature sphere congruence. When it does not degenerate, $\hat{f}$ envelops the same mean curvature sphere congruence as $f$.
\item (Vanishing theorem.) One can construct holomorphic forms on the underlying Riemann surface; on $S^2$ such forms always vanish identically.
\end{enumerate}

When the codimension is arbitrary, Ejiri generalized the construction of the conformal Gauss map for surfaces in $\mathbb{S}^n$. The same as before, this map is harmonic if and only if the original surface is Willmore. However, the duality theorem fails in general, since the second envelope of the mean curvature spheres might not exist \cite{Ejiri, BFLPP, ma2}. Fortunately, in codimension-two case the global assumption of being a 2-sphere guarantees the duality property. So in $\mathbb{S}^4$ one can still use the same method to obtain a classification as mentioned above.

Ejiri restricted to consider the subclass of Willmore surfaces in higher codimensional space for which the duality theorem still holds. This class is called \emph{S-Willmore surfaces} (see Section~3.1 for the definition and examples). He established a classification of S-Willmore 2-spheres in $\mathbb{S}^n$ as below. \\

\noindent
{\bf Theorem}~\cite{Ejiri}~~An immersed S-Willmore $S^2$ in $\mathbb{S}^n$ is M\"obius equivalent to one of the following surfaces
\begin{enumerate}
\item a minimal surface in $\mathbb{R}^n$ with embedded flat ends,
\item  a super-Willmore surface fully immersed in $\mathbb{S}^{2m}$. (It corresponds to a holomorphic, totally isotropic curve in an odd-dimensional complex projective space $\mathbb{C}P^{2m+1}$).
\end{enumerate}
In the end of his 1988 paper \cite{Ejiri}, Ejiri asked whether a Willmore two-sphere in $S^n$ must always be S-Willmore. If this is true, the classification of Willmore 2-spheres is finished. This problem remained open for a long time until a negative answer appeared in 2013 in \cite{DoWa1}.

Based on the celebrated DPW method for harmonic maps into non-compact symmetric space, Dorfmeister and the third author \cite{DoWa1} constructed a totally isotropic Willmore two-sphere in $S^6$ explicitly, which is full and not S-Willmore (see Section~5.1 in this paper for details). Moreover, in a follow-up work \cite{Wang-1}, \cite{Wang-3}, Wang provided a coarse classification for all Willmore two-spheres in $S^6$. According to their prediction, there also exist Willmore two spheres in $S^5$ which are not S-Willmore. But the computation along this route is very complicated, which makes it hard to derive an explicit expression or to discuss their global properties.\\

On the other hand, the \emph{adjoint transform} mentioned above is defined for \emph{any} immersed Willmore surface in $S^n$ \cite{ma2}. Such transforms produce new Willmore surfaces which always exist locally (however, not unique in general). This is a natural generalization of the dual Willmore surface as well as the 2-step B\"acklund transforms of Willmore surfaces in $S^4$ \cite{BFLPP}.

When the global assumption of being 2-sphere is imposed on a Willmore surface which is not totally isotropic (see Section~3 for a precise definition), it was noticed \cite{ma0} that one can construct a sequence of adjoint Willmore surfaces in a canonical way, with increasing isotropy order. Then the sequence has to terminate with a Euclidean minimal surface, whose next canonical adjoint transform degenerates to a single point. This picture mimics the famous harmonic sequence construction as well as the Willmore sequence produced by using B\"acklund transforms \cite{Leschke}. By analyzing the behavior of their Willmore sequence, in \cite{Leschke} Leschke and Pedit gave a new proof of the classification theorem of Willmore tori in $S^4$ with nontrivial normal bundle (which is quite similar to the case of Willmore two-spheres in $S^4$).

To overcome the aforementioned difficulty caused by the possible branch points, we construct some new conformal minimal branched immersion $\xi$ from $S^2$ to the accompanied de Sitter space $S^6_1$. Then one can use the following result \cite{Ejiri-i}.

\begin{proposition}\label{prop-xi-harmonic}
Any conformal harmonic map $\xi:S^2\to S^n_1\subset \mathbb{R}^{n+1}_1$ is totally isotropic, i.e., $\<\xi_z^{(k)},\xi_z^{(l)}\>\equiv 0,~\forall~k,l\in \mathbb{Z}^+.$
It is located in a subspace orthogonal to a non-zero light-like vector.
\end{proposition}

The proof of the classification theorem is divided into two cases, treated separately in Section~4 and 5, depending on whether the Hopf differential is isotropic or not. Note that one still needs to deal with the technical difficulty of singularities and to make sure that $\xi:S^2\rightarrow S^{6,1}$ is globally defined. We address this problem in the Appendix. For Willmore surfaces in $S^6$ or higher codimension, it stays an open problem to us how to guarantee the global existence of adjoint transforms.

The conformal harmonic map $\xi$ to the de Sitter space as above is similar to the conformal Gauss map first introduced in \cite{bryant1}. The main difference here is that this harmonic map $\xi$ is derived from an adjoint transform operation. This map $\xi$ also relates the geometry of Willmore surfaces with the integrable system theory (\cite{Helein}, \cite{BFLPP}, \cite{bpp}, \cite{DoWa1}, \cite{Wang-1}, \cite{Wang-3}). A very interesting observation is that Willmore surfaces are related with several harmonic maps into different spaces; see \cite{Helein} and \cite{ma2} for examples. This might provide some insight for the future work.

 In this connection we would like to mention the quantization theorem of the Willmore functionals. For an inner non-compact symmetric space $G/K$, there always exists a
compact dual symmetric space $U/H$, (see for example Section 4.4 of \cite{DoWa1}). In \cite{DoWa1}, Dorfmeister and Wang showed that for every harmonic map from a Riemann surface $M$ into
$G/K$, there exists a dual harmonic map into $U/H$. A basic observation of Burstall \cite{burstall2} shows that the energy of these two harmonic maps differ by an exact form. So when $M$ is closed, these two harmonic maps share the same energy. Applying to a Willmore two-sphere, its Willmore energy is equal the energy of its conformal gauss map, which is mapped into $SO(1,n+1)/SO(1,3)\times SO(n-2)$. The dual compact symmetric space of  $SO(1,n+1)/SO(1,3)\times SO(n-2)$ is $SO(n+2)/SO(4)\times SO(n-2)$. So its Willmore energy is equal to the energy of some harmonic two-sphere into $SO(n+2)/SO(4)\times SO(n-2)$.
 By \cite{burstall3}, the energy of a harmonic two-sphere into $SO(n+2)/SO(4)\times SO(n-2)$ is $2\pi k$ for some $k\in \mathbb{Z}^+\cup\{0\}$.
As a corollary, the Willmore functional of a Willmore 2-sphere in $S^n$ is also  $2\pi k$ for some $k\in \mathbb{Z}^+\cup\{0\}$. Till now, this result can not be derived from the method of adjoint transforms. We also point out that in \cite{Montiel}, Montiel showed that the energy of a Willmore 2-sphere in $S^4$ is $4\pi k$, for some $k\in \mathbb{Z}^+\cup\{0\}$ and all $k$ can be achieved by some Willmore immersion in $S^4$. We conjecture that this result still holds in $S^n$ when $n>4$. It seems that one need new insights to prove this.

Another interesting question is whether there exist new examples in the third class which are \emph{immersed (without any branch points)}. To answer it, we find out \emph{explicitly all} possible adjoint transforms of \emph{any} superconformal minimal surface in $\mathbb{R}^n$ in Section~6. Then in Section~7 we show how to guarantee the vanishing of branched points (including a discussion of the end behavior). An interesting by-product is the relationship with the classical construction of \emph{pedal surfaces}, which also appeared recently in the work of Dajczer and Vlachos \cite{Da-V2} on superconformal surfaces. Finally in Section~8 we describe three examples of the third class (immersed Willmore two-spheres which are not S-Willmore) in details.

It is worth mentioning that adjoint transforms in $S^n$ contrast sharply with 2-step B\"acklund transforms in $S^4$ in the following aspect. It has been shown in \cite{Leschke} that the latter preserves the smoothness of conformal Gauss map (the mean curvature sphere congruence). However in Section~7 (Remark~\ref{rem-end-gaussmap}) we will see that adjoint transforms may destroy the smoothness of the conformal Gauss map when $n\ge 5$.\\

This paper is organized as below. In Section~2 we review the surface theory in M\"obius geometry. The definition of S-Willmore surfaces and the adjoint transforms are included in Section~3. The proof of the classification theorem is divided into two cases, treated separately in Section~4 and 5, depending on whether the Hopf differential is isotropic or not. In Section~6, we provide a concrete description of all adjoint transforms of a superconformal minimal surface in $\R^n$. Section 7 devotes to the discussions of the immersion property of the adjoint surfaces, which have a close relationship with the end behavior of the original superconformal minimal surface in $\R^n$. We end this paper by providing a new Willmore two-sphere in $S^5$ as an adjoint surface of a superconformal minimal surface in $\R^5$ in Section~8 and proving a technical lemma in the appendix.  \\

\textbf{Acknowledgement}~~
The authors are grateful to NSFC for the continual support (the grant 10771005, 11171004, 11201340 and 11331002) of this
long-term research program. We thank F. Pedit for his interest in this work and for communicating to us the result of Burstall. The third named author is thankful to F. Burstall, J. Dorfmeister and F. Pedit for communications and discussions on the energy of harmonic maps and Willmore two-spheres.

\section{Surface theory in M\"obius geometry}

In this section, we will briefly review the surface theory in M\"obius geometry. For details see \cite{bpp, ma2}.

We identify the unit sphere $S^n\subset \R^{n+1}$ with the projectivized light cone via
\[
S^n\cong \mathbb{P}(\L):~~y\leftrightarrow [(y,1)]=[Y],
\]
where $\L\subset \R^{n+2}_1$ is the lightcone in the $(n+2)$-dimensional Lorentz space with the quadratic form
$\<Y,Y\>=-Y_{n+2}^2+\sum_{i=1}^{n+1}Y_i^2$ and $[(y,1)]$ is the homogeneous coordinate. The projective action of the Lorentz group on $\mathbb{P}(\L)$ yields all conformal diffeomorphisms of $S^n$. The following correspondence is well-known:
\begin{itemize}
\item A point $y\in S^n$ $~~~~\leftrightarrow~~~$ a lightlike line $[Y]\in \mathbb{P}(\L)$;
\item A $k$-dim sphere $\sigma\subset S^n$ $~~\leftrightarrow~~~$ a space-like $(n-k)$-dim subspace $\Sigma\subset \R^{n+2}_1$;
\item The point $y$ locates on the sphere $\sigma$ $~~~\leftrightarrow~~~$ $Y\bot\Sigma$.
\end{itemize}
For a conformal immersion $y:M^2\to S^n$ of a Riemann surface $M^2$, a local lift is just a map $Y$ from $M^2$ into the light cone $\L$ such that the null line spanned by $Y(p)$ is $y(p)$, $p\in M$.
Taking derivatives with respect to a local complex coordinate $z$,
we see that $\<Y_z,Y_z\>=0$ and $\<Y_z,Y_{\zb}\> >0$, since $y$ is a conformal immersion.

Before describing a specific choice of the moving frame, we briefly review the notion of mean curvature sphere. There is a 4-dimensional Lorentz subspace of $\R^{n+2}_1$ defined at every point of $M^2$ given by
\[
V={\rm span} \{ Y,\Re (Y_z),\Im (Y_z), Y_{z\zb}, \}
\]
which is independent to the choice of local lift $Y$ and complex coordinate $z$. Under the correspondence given above, $V$ describes a M{\"o}bius invariant geometric object, called the \emph{mean curvature sphere} of $y$. This name comes from the property that it is the unique 2-sphere tangent to the surface and having the same mean curvature vector as the surface at the tangent point when the ambient space is endowed with a metric of some Euclidean space (or any other space form). The corresponding map from $M$ into the Grassmannian $Gr(4,\R^{n+1,1})$ (which consists of 4-dimensional Lorentz subspaces) is the so-called \emph{conformal Gauss map} \cite{bryant1, Ejiri}.

For a given local coordinate $z$, there is a canonical
lift determined by $\abs{\d Y}^2=\abs{\d z}^2.$ We will always assume that $Y$ is such a canonical lift unless stated elsewhere.
Then a canonical frame of $V\otimes\C$ is given as
\begin{equation}
\label{frame}
\{Y,Y_z,Y_{\zb},N\},
\end{equation}
where we choose the unique $N\in V$ with
$\< N,N\>=0,~\< N,Y\>=-1,~\< N,Y_z\> =0.$
These frame vectors are orthogonal to each other except that
$\< Y_{z},Y_{\zb} \> = \frac{1}{2},~\< Y,N \> = -1.$ Let $\xi\in\Gamma(V^{\bot})$ be an arbitrary section of
the normal bundle $V^{\bot}$; $D$ is the normal connection.
The structure equations are as below:
\begin{equation}
\label{mov-eq}
\left\{
\begin{array}{llll}
Y_{zz} &=& -\frac{s}{2} Y + \kappa, \\[.2cm]
Y_{z\zb} &=&
-\<\kappa,\bar{\kappa}\> Y + \frac{1}{2}N, \\[.2cm]
N_{z} &=& -2 \<\kappa,\bar{\kappa}\> Y_z
- s Y_{\zb}+ 2D_{\zb}\kappa, \\[.2cm]
\xi_z &=& D_z\xi + 2\<\xi,D_{\zb}\kappa\> Y
- 2\<\xi,\kappa\> Y_{\zb}.
\end{array}
\right.
\end{equation}
The first equation among them is a fundamental one, which defines
two basic M{\"o}bius invariants associated with
the surface $y:M\to S^n$. The \emph{Schwarzian} $s$ is a complex valued function interpreted as a generalization of the usual Schwarzian derivative of a complex function. The section $\kappa\in \Gamma(V^{\bot}\otimes\C)$ may be identified with the normal-valued Hopf differential up to scaling.
Later we will need the fact that $\kappa$ vanishes exactly at the umbilic points, and that $\kappa\frac{\d z^2}{\abs{\d z}}$
is a globally defined differential form. See \cite{bpp} for more details.

The conformal Gauss, Codazzi and Ricci equations
as integrability conditions are
\begin{eqnarray}
& \frac{1}{2}s_{\zb} = 3\< D_z\bar{\kappa},\kappa\>
+ \<\bar{\kappa},D_z\kappa\>, \label{gauss}\\[.1cm]
& \Im ( D_{\zb}D_{\zb}\kappa
+ \frac{\bar{s}}{2}\kappa ) = 0, \label{codazzi}\\[.1cm]
& R_{\zb z}^D\xi
:= D_{\zb}D_z\xi - D_z D_{\zb}\xi
= 2\<\xi,\kappa\>\bar{\kappa}
- 2\<\xi,\bar{\kappa}\>\kappa. \label{ricci}
\end{eqnarray}

\section{Willmore surfaces and adjoint transforms}
There is a well-defined metric over $M$ invariant under M\"obius transformations, which is also conformal to the original metric induced from $y:M\to S^n$, called the \emph{M\"obius metric}:
\[
e^{2\omega}\abs{\d z}^2=4\<\kappa,\bar\kappa\>\abs{\d z}^2.
\]
It is well known that this metric is induced from the conformal Gauss map. The area of $M$ with respect to the M\"obius metric
\[
W(y):=2{\rm i}\cdot\int_M \abs{\kappa}^2 \d z\wedge\d\zb
\]
is exactly the famous \emph{Willmore functional}.
It coincides with the usual definition
$\widetilde{W}:=\int_M (H^2-K)\d M$
for an immersed surface in $\R^3$ with mean curvature $H$ and Gauss curvature $K$. A critical surface with respect to $W$ is called a \emph{Willmore surface}. In terms of the above invariants, Willmore surfaces are characterized by the
\emph{Willmore equation} \cite{bpp}:
\begin{equation}
\label{willmore}
D_{\zb}D_{\zb}\kappa+\frac{\bar{s}}{2}\kappa = 0.
\end{equation}
Note that this is stronger than the Codazzi equation \eqref{codazzi}.

\begin{remark}\label{harmonic}
As mentioned in the introduction, another important characterization of the Willmore surfaces is that the conformal Gauss map is harmonic \cite{blaschke, bryant1, Ejiri}.
This link motivated the study of Willmore surfaces as an integrable system \cite{Helein, DoWa1}. Although we will not pursue this line here, we would like to emphasize that in the proof of Theorem \ref{main}, the key step is the construction of some new conformal harmonic map $\xi$ into the de-Sitter space $S^6_1$, which should be viewed as derived from this harmonic conformal Gauss map.
\end{remark}

\begin{remark}\label{analytic}
Since Willmore surfaces satisfy an elliptic equation, Morrey's result (see \cite{bryant1} and Lemma~1.4 of \cite{Ejiri}) guarantees that the related geometric quantities are real analytical.
\end{remark}

\subsection{S-Willmore surfaces}

In the codim-1 case,
Bryant \cite{bryant1} noticed that when $y:M^2\to S^3$ is Willmore, there is always a dual conformal Willmore surface $\hat{y}=[\widehat{Y}]:M^2\to S^3$ enveloping the same mean curvature spheres. Note that in this case, except the umbilic points,
$D_{\zb}\kappa$ depends linearly on
$\kappa$, thus locally there is a function $\mu$ such that
\begin{equation}
\label{swillmore}
D_{\zb}\kappa +\frac{\bar\mu}{2}\kappa =0.
\end{equation}

When the codimension is higher, Ejiri first noticed that \eqref{swillmore} does not hold for all Willmore surfaces.
He observed moreover that a surface satisfying \eqref{swillmore} for some $\mu$ is automatically a Willmore surface with a dual Willmore surface. Conversely, the duality theorem holds true if the first Willmore surface $y$ satisfies \eqref{swillmore} \cite{Ejiri, ma1}.
So Willmore surfaces satisfying \eqref{swillmore} for some $\mu$ (locally defined, depending on the coordinate $z$) provide exactly all Willmore surfaces having a dual surface. They are called \emph{S-Willmore surfaces}.
This special class includes Willmore surfaces in $S^3$, superconformal surfaces in $S^4$,
and minimal surfaces in space forms $\R^n, S^n, H^n$.

\begin{remark}
Note that the above definition of S-Willmore surfaces includes the codimension one Willmore surfaces. This is slightly different from Ejiri's original definition \cite{Ejiri}, for the purpose of including all Willmore surfaces with dual surfaces.
\end{remark}

The dual surface $\hat{y}=[\widehat{Y}]:M\to S^n$ may be written down explicitly as
\begin{equation}
\label{yhat}
\widehat{Y}=\frac{1}{2}\abs{\mu}^2 Y
+\bar\mu Y_z +\mu Y_{\zb} +N
\end{equation}
with respect to the frame $\{Y,Y_z,Y_{\zb},N\}$.
Calculation using \eqref{mov-eq} and
\eqref{swillmore}, \eqref{willmore} yields
\begin{equation}
\label{yhat-z}
\widehat{Y}_z=\frac{\mu}{2}\widehat{Y}
   + \rho\left(Y_z + \frac{\mu}{2}Y\right),~~
   \text{where}~\rho:=\bar\mu_z-2\<\kappa,\bar\kappa\>.
\end{equation}
Here $\rho\abs{\d z}^2$ is a globally defined invariant
associated with $\{y,\hat{y}\}$ \cite{ma2}.
It follows
\begin{equation}
\label{yhat-metric}
\<\widehat{Y}_z,\widehat{Y}_z\>=0, \quad
\<\widehat{Y}_z,\widehat{Y}_{\zb}\>
=\frac{1}{2}\abs{\rho}^2.
\end{equation}
It is straightforward to verify that $\hat{y}$ share
the same mean curvature sphere as $y$, at the points where $\hat{y}$ is immersed. By the characterization mentioned in Remark~\ref{harmonic}, $\hat{y}$ is also a conformal Willmore immersion into $S^n$ when $\rho\ne 0$.

\begin{remark}\label{rem-minimal}
When $\rho\equiv 0$, by \eqref{yhat-z}, we know $\widehat{Y}$ corresponds to a fixed point $[\widehat{Y}]$ in $S^n$. Up to a M\"obius transformation we may assume that $[\widehat{Y}]$ is $\infty$, the point at infinity of the Euclidean space $\R^n$. Since the mean curvature spheres of $y=[Y]$ always pass through this point, these spheres are all planes. According to the geometric meaning of the mean curvature spheres given before, the surface has mean curvature $H=0$ everywhere in $\R^n$. So $y$ is M\"obius equivalent to a Euclidean minimal surface.
\end{remark}

\subsection{The adjoint transform}

In higher dimensional sphere $S^n$, a Willmore surface does not necessarily have a dual surface. Therefore
one needs to generalize the notion of dual Willmore surfaces. In \cite{ma0,ma2} the first author introduced adjoint transforms for arbitrary Willmore surfaces .

Given a Willmore surface $y:M^2\to S^n$, its adjoint transform is another (branched) conformal immersion $\hat{y}:M^2\to S^n$ such that the corresponding point $\hat{y}(z)$ locates on the same mean curvature sphere of $y$ at $z$, at the same time $\hat{y}$ \emph{half-touches} this sphere \cite{ma0, ma2}.

Since the corresponding point $\hat{y}=[\widehat{Y}]$ still locates on the same mean curvature sphere as $y$, we have similar equations as \eqref{yhat} and \eqref{yhat-z}, depending on a local function $\mu$:
\begin{equation}
\label{yhat2}
\widehat{Y}=\frac{1}{2}\abs{\mu}^2 Y
+\bar\mu Y_z +\mu Y_{\zb} +N,
\end{equation}
\begin{equation}
\label{yhat-z2}
\widehat{Y}_z=\frac{\mu}{2}\widehat{Y}
   + \rho\left(Y_z + \frac{\mu}{2}Y\right)+2\eta,
\end{equation}
where
\begin{equation}\label{eq-eta}
\rho:=\bar\mu_z-2\<\kappa,\bar\kappa\>,~~
\eta:=D_{\zb}\kappa +\frac{\bar\mu}{2}\kappa.
\end{equation}
To satisfy the \emph{half-touching} condition and the conformal condition \cite{ma2} we require that
\begin{align}
\theta:&=\mu_z-\frac{1}{2}\mu^2-s=0, \label{eq-theta}\\
 \<\eta,\eta\>&=\frac{\bar\mu^2}{4}\<\kappa, \kappa\>
+\bar\mu \<\kappa,D_{\bar{z}}\kappa\>
 +\<D_{\bar{z}}\kappa,D_{\bar{z}}\kappa\>=0. \label{eq-theta-2}
\end{align}
The formulas below \cite{ma2} are useful later and easy to verify:
\begin{equation}\label{eq-eta-zb}
D_{\zb}\eta=\frac{\bar\mu}{2}\eta,~~~
\rho_{\zb}=\bar\mu\rho+4\<\eta,\bar\kappa\>.
\end{equation}
A basic fact about the adjoint transform is the following theorem.

\begin{theorem}
\label{thm-adjoint-dual}~\cite{ma2}~~The adjoint transform $[\widehat{Y}]:M^2\to S^n$ is a Willmore surface. Conversely, the original Willmore surface $[Y]$ is an adjoint transform of $[\widehat{Y}]$.
\end{theorem}

As to the uniqueness problem of adjoint transforms, we define a $6$-form using the discriminant of the quadratic equation about $\bar\mu$ \eqref{eq-theta-2} as below \cite{ma0,mawang}:
\begin{equation}\label{eq-Theta2}
\Theta_0=\left[\<D_{\zb}\kappa,\kappa\>^2-\<D_{\zb}\kappa,
D_{\zb}\kappa\>\<\kappa,\kappa\>\right](\d z)^6.
\end{equation}
It is straightforward to verify the independence to the choice of coordinate $z$. Moreover, by the Willmore condition, this form is holomorphic.

In general, a surface is called \emph{superconformal} if $\kappa$ itself is isotropic; it is called \emph{totally isotropic} if $\kappa$ and its derivatives are all isotropic. When $\kappa$ is isotropic, i.e.,  $\langle\kappa,\kappa\rangle\equiv 0$, the Willmore condition \eqref{willmore} implies $\<\kappa,D_{\bar{z}}\kappa\>=\<D_{\bar{z}}\kappa,D_{\bar{z}}\kappa\>=0$. Therefore $\Theta_0$ vanishes in this case. Now we can state

\begin{theorem}\label{thm-adjoint-no}~\cite{ma2}
\begin{enumerate}
\item
  If $\langle\kappa,\kappa\rangle\equiv0$, then any solution to \eqref{eq-theta} defines an adjoint surface of $y$ via \eqref{yhat2}.
\item
 If $\langle\kappa,\kappa\rangle\not\equiv0$ and $\Theta_0\ne 0$, then there exists exactly two solutions to \eqref{eq-theta-2}. They provide two adjoint surfaces of $y$ via \eqref{yhat2}. Moreover, $y$ is not S-Willmore.
\item
 If $\langle\kappa,\kappa\rangle\not\equiv0$ and $\Theta_0$ vanishes, then there exists exactly one solution to \eqref{eq-theta-2}. There is a unique adjoint surface of $y$ via \eqref{yhat2}. Especially, if $y$ is S-Willmore, then the unique adjoint surface is its dual surface.
\end{enumerate}
\end{theorem}

\section{Superconformal Willmore 2-spheres in $S^5$}

To prove the classification theorem~\ref{main}, we consider two cases separately:
\begin{enumerate}
\item The \emph{Superconformal} case: $\langle \kappa,\kappa\rangle\equiv 0$ identically;
\item The \emph{Non-superconformal} or \emph{non-isotropic} case: $\langle \kappa,\kappa\rangle\ne 0$ on an open dense subset.
\end{enumerate}
We may also restrict to consider only the case when it is a \emph{full} immersion in $S^5$, i.e., the surface is not contained in any lower dimensional sphere (plane). This excludes the superconformal surfaces in $S^4$.

\begin{proposition}\label{isotropic0}
Any superconformal Willmore surface $y:M\to S^5$ is S-Willmore. Moreover:\\
(1) The dual Willmore surface $\tilde{Y}$ is also superconformal and S-Willmore at its regular points.\\
(2) If $y$ is full in $S^5$, then any other adjoint transform $\widehat{Y}$ of $Y$ is not S-Willmore and also non-superconformal.
\end{proposition}

The first conclusion of this proposition was contained in \cite{Da-V} by Dajczer and Vlachos. Here we give an independent and simple proof, which makes the full picture more clear.

\begin{proof}
Differentiating $\<\kappa,\kappa\>=0$ twice and using the Willmore condition, we know
\[
\langle \kappa,D_{\bar{z}}\kappa\rangle=\langle D_{\bar{z}}\kappa,D_{\bar{z}}\kappa\rangle=0.
\]
In the non-trivial case $\kappa\not\equiv 0$, due to the restriction on the codimension, the isotropic sub-bundle $Span\{\kappa,D_{\bar{z}}\kappa\}$ has complex rank $1$. Hence $D_{\bar{z}}\kappa\parallel \kappa$. So this is a S-Willmore surface.

The dual surface $\widetilde{Y}$ of $Y$ is defined in \eqref{yhat} by $\mu$ satisfying $\eta=D_{\bar{z}}\kappa+\frac{\bar{\mu}}{2}\kappa\equiv 0$.
As a consequence, from \eqref{yhat-z},
$\widetilde{Y}_{zz}=0\mod \{\widetilde{Y}, \widetilde{Y}_z, \kappa\}.$
Hence  $\widetilde{Y}$ is still superconformal. It is obviously a S-Willmore surface with a dual Willmore surface $Y$.

For any other adjoint transform $\widehat{Y}$, now $\eta=D_{\bar{z}}\kappa+\frac{\bar{\mu}}{2}\kappa$ is not identically zero. $\eta$ is parallel to $\kappa$ since $Y$ is S-Willmore. So $D_{z}\eta$ is not parallel to $\kappa$, otherwise $Y$ will be contained in $\hbox{Span}\{Y,Y_z,Y_{\bar{z}}, Y_{z\bar{z}},\kappa,\bar\kappa\}$ and hence not full in $S^5$. Since the co-dimension is $3$, we have $\<D_{z}\eta,D_{z}\eta\>\neq 0$ on an open subset. By the same computation we see $\widehat{Y}_{zz}=D_{z}\eta\mod \{\widehat{Y}, \widehat{Y}_z, \eta\}.$ Hence $\<\widehat{Y}_{zz},\widehat{Y}_{zz}\>\neq 0$. So the Hopf differential of $\widehat{Y}$ is not isotropic.

From this fact and Theorem \ref{thm-adjoint-no}, $\widehat{Y}$ has a unique adjoint transform. Since $Y$ is an adjoint surface of $\widehat{Y}$, $Y$ is the unique one of $\widehat{Y}$. If $\widehat{Y}$ is S-Willmore, then its dual surface, as an adjoint transform, has to coincide with $Y$. Thus $\{Y,\widehat{Y}\}$ forms a pair of dual Willmore surface. This contradicts with the uniqueness of the dual Willmore surface and the assumption that $\widehat{Y}\ne \widetilde{Y}$.
\end{proof}

The result below follows directly from Proposition~\ref{isotropic0} and Ejiri's classification of S-Willmore 2-spheres \cite{Ejiri}. Here we provide an independent proof, not only for the sake of being self-contained, but also because that the conformal harmonic map $\xi:M\to S^6_1$ appearing in this proof is interesting in its own right.  A similar construction, i.e., finding out some globally defined conformal map $\xi:M\to S^6_1$, plays a crucial role in the discussion of the non-isotropic case in the next section.
\begin{proposition}\label{isotropic}
A full superconformal Willmore immersion $y:S^2\to S^5$ must be M\"obius equivalent to a superconformal minimal surface in $\mathbb{R}^5$.
\end{proposition}
\begin{proof}
By Proposition~\ref{isotropic0}, $y$ is an S-Willmore surface. We may write $D_{\zb}\kappa=-\frac{\bar\mu}{2}\kappa$, and the dual Willmore surface is given by \eqref{yhat}. Here we have assumed that it is not totally umbilic.

By Lemma~\ref{lemma-subbundle} in the appendix (see also \cite{Ejiri}), the holomorphic isotropic line-bundle $\mathrm{Span}\{\kappa\}$ is defined on the whole $S^2$. This fact makes it possible to choose a real unit normal vector $\xi$ such that
\[\xi\in V^{\bot},~~\xi~\bot~ \kappa,~\bar\kappa.\]
It is evident that $\xi$ can be chosen globally and consistently. Notice that
$\<D_z\xi,\bar\kappa\>=-\<\xi,D_z\bar\kappa\>
=\<\xi,\frac{\mu}{2}\bar\kappa\>=0,~ \<D_z\xi,\xi\>=0,~\<\xi,\kappa\>=0.$ Then the structure equation
tells us that
\[
\xi_z=D_z\xi-2\langle\xi,\kappa\rangle (Y_{\bar{z}}+\frac{\bar\mu}{2}Y)=D_z\xi=\lambda \bar\kappa
\]
for some function $\lambda$. It follows
$\<\xi_z,\xi_z\>=0$, i.e, $\xi:S^2\to S^6_1$ is a conformal map.

Differentiating once more, one can show that $\xi_{z\bar{z}}\in V^{\bot}$ is a real normal bundle section by the structure equations and $\<\bar\kappa,\bar\kappa\>=0=\<\bar\kappa,D_z\bar\kappa\>$.
Since $\xi_{z\bar{z}}~\bot~\xi_z$ and $\xi_z\parallel \bar\kappa$, there must be $\xi_{z\bar{z}}\parallel \xi$. Thus
$\xi:S^2\to S^6_1$ is a branched conformal harmonic map.

According to Proposition~\ref{prop-xi-harmonic} \cite{Ejiri-i}, $\xi_{zzz}$ must be isotropic on $S^2$. Direct computation shows
\begin{align*}
\xi_{zz}&=(\cdots)\bar\kappa-2\langle \xi_z,\kappa\rangle
(Y_{\bar{z}}+\frac{\bar\mu}{2}Y),\\
\xi_{zzz}&=(\cdots)\bar\kappa+(\cdots)(Y_{\bar{z}}+\frac{\bar\mu}{2}Y)
-\langle \xi_z,\kappa\rangle(\rho Y+\widehat{Y}),\\
\Longrightarrow~~~
0&\equiv \langle\xi_{zzz},\xi_{zzz}\rangle=
-2\rho\langle \xi_z,\kappa\rangle^2.
\end{align*}
If $\langle \xi_z,\kappa\rangle=0$ on an open subset of $S^2$, the real analytical property forces it to be zero identically. Because $\xi_z=\lambda\bar\kappa$ and $\kappa\ne 0$ on an open dense subset, there follows $\xi_z\equiv 0$. So $\xi$ is a constant unit vector in $\mathbb{R}^7_1$. In this case the original surface $y$ is a superconformal surface in $S^4$, a contradiction to the fullness assumption.
The other possibility is $\rho\equiv 0$. According to Remark~\ref{rem-minimal},  $y$ is M\"obius equivalent to a minimal surface in $\mathbb{R}^5$. This finishes the proof.
\end{proof}

\section{Non-superconformal Willmore 2-spheres in $S^5$}

The classification theorem~\ref{main} follows from Proposition~\ref{isotropic} and Proposition~\ref{thm-main3} below.

\begin{proposition}\label{thm-main3}
Let $y: S^2\to S^5$ be a full Willmore immersion with $\<\kappa,\kappa\>\not\equiv 0$. Then $y$ belongs to one of the following two cases:
\begin{enumerate}
\item $y$ is M\"obius equivalent to a minimal surface in $\mathbb{R}^5$ when it is S-Willmore;
\item $y$ is an adjoint transform of a  superconformal minimal surface in $\mathbb{R}^5$ ($y$ is not S-Willmore).
\end{enumerate}
\end{proposition}

We first recall the following theorem \cite{ma0}.
\begin{theorem}\label{thm-uniqueness}
Let $y:S^2\to S^n$ be a Willmore immersion and not superconformal. Then its adjoint transform is unique and defined globally on $S^2$ as a branched conformal immersion. When $y$ is S-Willmore, this adjoint transform is exactly the dual surface.
\end{theorem}
\begin{proof}
This is the direct corollary of case (3) of Theorem~\ref{thm-adjoint-no} and the vanishing theorem of holomorphic forms on $S^2$.

Here this unique adjoint transform is given by \eqref{yhat2};
the unique solution $\bar\mu$ is given by
\begin{equation}\label{eq-mu2}
\bar\mu=-\frac{\<\kappa,\kappa\>_{\zb}}{\<\kappa,\kappa\>}
\end{equation}
when $\<\kappa,\kappa\>\ne 0$. Define $\eta=D_{\zb}\kappa+\frac{\bar\mu}{2}\kappa$
as in \eqref{eq-eta} and also recall that
$D_{\zb}\eta=\frac{\bar\mu}{2}\eta$
by the Willmore condition.
Notice that even at a zero of $\<\kappa,\kappa\>$, the limit of $\mu$ still exists according to Chern's lemma (see also \cite{Ejiri}, \cite{DoWa1}). Thus the line spanned by $\widehat{Y}$, which corresponds to a point in $S^n$, has a well-defined limit. In particular, when the limit of $\mu$ is $\infty$ at one point of $S^2$, the limit of the real line $[\widehat{Y}]$ is nothing but $[Y]$. As a conclusion, the adjoint transform $[\widehat{Y}]$ extends to a branched conformal immersion $S^2\to S^n$.
\end{proof}
From the proof of Theorem \ref{thm-uniqueness}, one observes that the adjoint surface is the dual surface of $y$ if and only if $\eta\equiv0$. We consider two cases separately depending on whether this holds true, which are treated separately in Proposition~\ref{prop-min} and Proposition~\ref{prop-not-min}.
\begin{proposition}\label{prop-min}
Let $y:S^2\to S^5$ be a full Willmore immersion and non-superconformal. When it is S-Willmore, $y$ is M\"obius equivalent to a minimal surface in $\mathbb{R}^5$.
\end{proposition}

\begin{proof}
We define $\bar\mu$ as above locally except the umbilical points. It follows from \eqref{eq-eta-zb} that $\rho_{\bar{z}}=\bar\mu\rho$. As a consequence, the $4$-form
\begin{equation}\label{eq-Theta3}
\Theta_3=\rho\<\kappa,\kappa\>(\d z)^4
\end{equation}
is well-defined and holomorphic everywhere except those umbilical points. We claim that $\Theta_3$ extends to the whole $S^2$ as a holomorphic form, i.e., those umbilical points are indeed removable singularities of this holomorphic form.

To show that, by Lemma~\ref{lemma-subbundle}, the line-bundle spanned by $\kappa$ is defined even at the zeros of $\kappa$. In a small neighborhood of any zero of $\kappa$, we can always take a non-zero holomorphic section of this line-bundle locally. Denote it by $\psi$ and write $\kappa=f\psi$. Obviously, this locally defined $f$ is smooth without any pole. It follows
\[
\bar\mu=-2\frac{f_{\bar{z}}}{f},
\]
and $\bar\mu f=-2f_{\bar{z}}$ is regular. So the singular term of $\rho\<\kappa,\kappa\>$ is
\[
\bar\mu_z\<\kappa,\kappa\>=\bar\mu_z \cdot f^2\<\psi,\psi\>
=(\bar\mu f^2)_z-2\bar\mu f f_z=-2(f_{\bar{z}}f)_z+4f_{\bar{z}}f_z.
\]
The final expression shows that poles do not occur. So $\Theta_3$ is a holomorphic form defined on $S^2$. It must vanish identically.
By the assumption that $\<\kappa,\kappa\>\ne 0$ on an open dense subset, we know $\rho\equiv 0$. So the original surface is M\"obius equivalent to a minimal surface in $\mathbb{R}^5$.
\end{proof}

\begin{proposition}\label{prop-not-min}
Let $y:S^2\to S^5$ be a full Willmore immersion which is not superconformal and not S-Willmore. Then it is an adjoint transform of a minimal surface in $\mathbb{R}^5$.
\end{proposition}

\begin{proof}
$\eta$ being not identically zero implies that $y$ is not S-Willmore.  Otherwise, the dual Willmore surface would be another adjoint transform of $y$, a contradiction with Theorem \ref{thm-uniqueness}.

We scale $\eta$ to get another isotropic section $\eta^\sharp$ without poles as below:
\[
\eta^\sharp:=\<\kappa,\kappa\>\eta=\<\kappa,\kappa\>D_{\zb}\kappa
-\<\kappa,D_{\zb}\kappa\>\kappa.
\]
It follows from the Willmore condition \eqref{willmore} and $\Theta_0=0$ that
\[
D_{\zb}\eta^\sharp=\<\kappa,\kappa\>\eta=
\<\kappa,D_{\zb}\kappa\>D_{\zb}\kappa
-\<D_{\zb}\kappa,D_{\zb}\kappa\>\kappa~\parallel~\eta^\sharp,
\]
\[
D_{\zb}D_{\zb}\eta^\sharp=-\frac{\bar{s}}{2}\eta^\sharp.
\]
Using Proposition~\ref{lemma-subbundle} once again, we know that
$\eta^\sharp$ spans an isotropic complex line-bundle on the whole $S^2$, which is also a holomorphic sub-bundle of the complex normal bundle.
$\eta$ can be regarded as a (local) section of this bundle on $M_0$.
The real and imaginary parts of all such sections span a rank-$2$ subbundle of the normal bundle, whose orthogonal complement has a well-defined global section $\xi$ on the whole $S^2$. Regarded as a frame vector, on $M_0$ it satisfies
\[
\xi~\bot~\eta,\bar\eta; ~~~\<\xi,\xi\>=1.
\]

We claim that $\xi:S^2\to S^6_1$ is a conformally harmonic map into the de-sitter space, i.e., it is a conformal minimal immersion on an open dense subset of $S^2$.

To show this, we still use $\eta,$ $\bar\eta$, $\xi$ as a local frame of the complex normal bundle. It follows from the Ricci equation \eqref{ricci} and $\<\xi,\eta\>=0$ that
\begin{equation}\label{eq-xi-z}
\xi_z=D_z\xi-2\<\xi,\kappa\>(Y_{\zb}+\frac{\bar\mu}{2}Y).
\end{equation}
Note that
\[
\<D_z\xi,\xi\>=0,~
\<D_z\xi,\bar\eta\>=-\<\xi,D_z\bar\eta\>
=-\<\xi,\frac{\mu}{2}\bar\eta\>=0,~~\Rightarrow~
D_z\xi\parallel \bar\eta.
\]
Hence the equation \eqref{eq-xi-z} can be re-written as
\begin{equation}\label{eq-xi-z2}
\xi_z=\lambda\bar\eta-2\<\xi,\kappa\>
(Y_{\zb}+\frac{\bar\mu}{2}Y)
\end{equation}
for some complex function $\lambda$ locally. As a consequence, \[
\xi\bot Y,Y_{\bar{z}},\hat{Y},\bar\kappa,\bar\eta;
~~~\<\xi_z,\xi_z\>\equiv 0.
\]
It follows that $\xi$ is a conformal mapping to $S^6_1$.

To show that $\xi$ is a minimal surface, we verify $\xi_{z\zb}\parallel \xi$. Since
\[
\xi^\bot=\mathrm{Span}_{\C}\{Y,Y_z,Y_{\zb}, N,\eta,\bar\eta\}
=\mathrm{Span}_{\C}\{Y,Y_z,Y_{\zb}, \widehat{Y},\xi_z,\xi_{\zb}\},
\]
and $\xi_{z\zb}$ is real vector-valued, one needs only to verify
$\xi_{z\zb}~~\bot ~~Y,Y_z,\widehat{Y},\xi_z$ as below:
\eqref{yhat-z2}:
\[\<\xi_{z\zb},\xi_z\>=\frac{1}{2}\<\xi_z,\xi_z\>_{\zb}=0;\]
\[\<\xi_{z\zb},Y\>=-\<\xi_z,Y_{\zb}\>=0;\]
\[\<\xi_{z\zb},Y_{\zb}\>=-\<\xi_z,Y_{\zb\zb}\>
=-\<\xi_z,\bar\kappa-\frac{\bar{s}}{2}Y\>=0;\]
\[\<\xi_{z\zb},\widehat{Y}\>=-\<\xi_z,\widehat{Y}_{\zb}\>=0.\]
Thus we have proved the claim. The conclusion of Proposition~\ref{prop-xi-harmonic} applies to branched conformal minimal surface $\xi:S^2\to S^6_1$. In particular the mapping $\xi$ is orthogonal to a constant light-like vector $Y^*$.

 To relate $\xi$ with some adjoint transform, we compute
\[\xi_{zz}
\mod \{\bar\eta, Y_{\bar{z}}+\frac{\bar\mu}{2}Y\},\]
which should be isotropic:
\begin{align}
\xi_{zz}&= -2\<D_z\xi,\kappa\>Y_{\bar{z}}
+2\<D_z\xi,D_{\bar{z}}\kappa\>Y\notag \\
&~~~~~~-2\<\xi,\kappa\>
\left(\frac{1}{2}N-\<\kappa,\bar\kappa\>Y
+\frac{\bar\mu}{2}Y_z+\frac{\bar\mu_z}{2}Y\right)
\qquad(\!\!\!\!\mod~~\bar\eta,Y_{\bar{z}}+\frac{\bar\mu}{2}Y~)\notag\\
&= 2\left(\<D_z\xi,\eta\>-\<\xi,\kappa\>\rho\right)Y
-\<\xi,\kappa\>\widehat{Y} \qquad\qquad\qquad~~~(\!\!\!\!\mod~~\bar\eta,Y_{\bar{z}}+\frac{\bar\mu}{2}Y~)\notag\\
&= -2\<\xi,D_z\eta+\frac{\rho}{2}\kappa\>Y
-\<\xi,\kappa\>\widehat{Y}+(\cdots)\bar\eta
+(\cdots)\left(Y_{\bar{z}}+\frac{\bar\mu}{2}Y\right).  \label{eq-xi-zz}
\end{align}
Note that $\langle \xi,\kappa\rangle$ is non-zero on an open dense subset (otherwise, suppose $\langle \xi,\kappa\rangle=0$ on an open subset, from $\kappa\bot \xi,\eta$ we deduce $\kappa\parallel\eta$, which contradicts with the fact that $\<\kappa,\kappa\>$ is not identically zero on any open subset).
Since $\xi_{zz}$ is isotropic, the coefficient of $Y$ in the expression above must vanish. In particular,
\begin{equation}\label{eq-xi-zz2}
\xi_{zz}\in \mathrm{Span}\{\bar\eta,Y_{\bar{z}}+\frac{\bar\mu}{2}Y, \widehat{Y}\},~~~\Rightarrow~~
0=\<\xi_{zz},\widehat{Y}\>=-\<\xi_z,\widehat{Y}_z\>.
\end{equation}

With these preparations, now we are able to discuss the geometry of
the adjoint transform $\hat{Y}$. First, the mean curvature sphere of
$\hat{Y}$ is given by
\begin{equation}\label{eq-yhat-sphere2}
\mathrm{Span}\{\hat{Y},\hat{Y}_z,\hat{Y}_{\zb},\hat{Y}_{z\zb}\}
~~\bot~~\{\xi,\xi_z,\xi_{\zb}\}.
\end{equation}
It is clear $\hat{Y}\bot ~\xi,\xi_z,\xi_{\zb}.$
By \eqref{yhat-z2} and \eqref{eq-xi-zz2} it follows $\hat{Y}_z\bot~ \xi,\xi_{\zb},\xi_z$. Finally, based on the fact $\xi_{z\zb}\parallel \xi$,
the following orthogonality conditions hold:
\[
\<\hat{Y}_{z\zb},\xi\>=-\<\hat{Y}_{\zb},\xi_z\>=-\<\hat{Y},\xi_{z\zb}\>=0,
\]
\[
\<\hat{Y}_{z\zb},\xi_z\>=-\<\hat{Y}_z,\xi_{z\zb}\>=0.
\]

As a consequence,
the fixed light-like vector $Y^*$, which is also orthogonal to the frames $\{\xi,\xi_z,\xi_{\zb}\}$, must be contained in the subspace
\[
\mathrm{Span}\{\hat{Y},\hat{Y}_z,\hat{Y}_{\zb},\hat{Y}_{z\zb}\}
\]
on an open dense subset of $S^2$. In other words, the mean curvature spheres of $[\widehat{Y}]$ pass through a fixed point $[Y^*]$. Taking this $[Y^*]$ as the point at infinity, $[\widehat{Y}]$ is (M\"obius equivalent to) a minimal surface in an affine $\mathbb{R}^5$.
The superconformality of $\hat{Y}$ is a corollary of (2) and (3) of Theorem \ref{thm-adjoint-no}.
The original Willmore 2-sphere $[Y]$ is an adjoint transform of this minimal surface by Theorem \ref{thm-adjoint-dual}.
This completes the proof.
\end{proof}

\section{Adjoint transforms of superconformal minimal surfaces in $\R^n$}

This section aims to derive all adjoint transforms of superconformal minimal surfaces in $\R^n$.
To this end, we will take a conformal complex coordinate $z$ and restrict to consider only \emph{local} theory. Since for $M^2=S^2=\mathbb{C}\cup\{\infty\}$ there always exists a local complex coordinate $z$, it is easy to see  how these discussions fit into \emph{global} context.

For convenience, we use dot product to denote the Euclidean inner product in $\R^n$, and $\<~,~\>$ to denote the Lorentz inner product of $\R^{n+2}_1$. They are extended to $\mathbb{C}$-bilinear products automatically when complex vectors are involved.

Consider a conformally immersed minimal surface $x:M^2\to \R^n$.
The map $x:M^2\to \R^n$ satisfies
\[
x_z\cdot x_z=0,~~x_{z\zb}=0,~~x_z\cdot x_{\zb}=\frac{1}{2}e^{2w}.
\]
Its classical normal valued Hopf differential is given by
\begin{equation}\label{eq-Q1}
Q\triangleq x_{zz}-2 w_z x_z
=x_{zz}-\frac{x_{zz}\cdot x_{\zb}}{x_z\cdot x_{\zb}} x_z.
\end{equation}
The Gauss equation and Codazzi equation are
\begin{equation}\label{eq-Q2}
e^{-2w}|Q|^2=w_{z\zb},~~D_{\zb} Q=0.
\end{equation}

Taking the inverse stereographic projection and then lifting to the lightcone, the canonical lift $X$ of $x$ has the form
\begin{equation}\label{eq-x}
X=e^{-w}\left(x,\frac{-1+x\cdot x}{2},\frac{-1-x\cdot x}{2}\right).
\end{equation}
It is easy to check that
\begin{equation}\label{eq-x1}
X_z=-w_z X+e^{-w}(x_z,x_z\cdot x,-x_z\cdot x),
\end{equation}
and $\<X_z,X_{\bar{z}}\>=\frac{1}{2}$.
Direct computation also yields
\begin{align}\label{eq-x2}
X_{z\bar{z}}&=(-w_{z\bar{z}}-w_z w_{\bar{z}})X
-w_z X_{\bar{z}}-w_{\bar{z}} X_z +\frac{e^w}{2}(0,1,-1),\\
X_{zz} &=(-w_{zz}+(w_z)^2)X +e^{-w}(Q,Q\cdot x,-Q\cdot x).
\end{align}
Comparing with \eqref{mov-eq} we obtain
\begin{equation}\label{eq-x3}
s=-2w_{zz}+2(w_z)^2,~~
\kappa=e^{-w}(Q,Q\cdot x,-Q\cdot x)
\end{equation}
and
\begin{equation}\label{eq-x4}
\frac{1}{2}N=X_{z\bar{z}}+\<\kappa,\bar\kappa\>X=
(e^{-2w}|Q|^2-w_{z\bar{z}}-w_z w_{\bar{z}})X
-w_z X_{\bar{z}}-w_{\bar{z}} X_z +\frac{e^w}{2}(\vec{0},1,-1).
\end{equation}
In particular, by the Codazzi equation (the second in \eqref{eq-Q2}) we have
\begin{equation}\label{eq-x5}
D_{\zb}\kappa=-\frac{\overline{\mu^*}}{2}\kappa, ~~~\text{where}~
\mu^*=2w_z.
\end{equation}
This confirms the fact that a Euclidean minimal surface is S-Willmore, whose dual surface $[X^*]$ degenerates to the single point $[(\vec{0},1,-1)]$, where
\begin{equation}\label{eq-x6}
X^*=N+\overline{\mu^*} X_z +\mu^* X_{\zb}+\frac{\abs{\mu^*}^2}{2} X=e^w(\vec{0},1,-1).
\end{equation}

\begin{remark}
Note that this factor $\mu^*$ is still defined at an umbilic point (the zero of $\kappa$) of a minimal surface in $\R^n$ (even for those in $S^n$ or $H^n$). For umbilics of generic S-Willmore surfaces this property may not hold true.
\end{remark}

From now on we assume that the original minimal surface $x$ is super-conformal, i.e.,
\[
0=x_{zz}\cdot x_{zz}=Q\cdot Q=\<\kappa,\kappa\>.
\]
As pointed out in Theorem~\ref{thm-adjoint-no}, in this situation, an adjoint transform of $x$ corresponds to a solution $\mu$ of equation~\eqref{eq-theta}
\[
\mu_z-\frac{1}{2}\mu^2-s=0,
\]
where the coefficient function $s$ is the previously defined Schwarzian locally given by \eqref{eq-x3}.
Such a Riccati equation is well-known to be related with another second order linear ordinary differential equation on the unknown $\zeta$:
\begin{equation}\label{eq-mu1}
   \zeta_{zz}=-\frac{s}{2}\zeta.
\end{equation}
The correspondence between solutions to these two equations is given by
\begin{equation}\label{eq-mu2}
\mu=\frac{-2\zeta_z}{\zeta}.
\end{equation}
That means we can find a solution $\mu$ to the first equation~\eqref{eq-theta} from a solution $\zeta$ of the second equation~\eqref{eq-mu1} using the formula~\eqref{eq-mu2}; conversely, any solution $\mu$ to \eqref{eq-theta} is constructed in this way. Thus the problem is reduced to solve \eqref{eq-mu1}.

Note that there is already a special solution $\zeta^*=e^{-w}$ to the equation~\eqref{eq-mu1}, corresponding to the known solution $\mu^*=2w_z$ for \eqref{eq-theta}.

As pointed out in \cite{ma2}, we can find all general solutions $\zeta$ to the equation~\eqref{eq-mu1} if a special solution $\zeta^*$ is known. Write $\zeta=\lambda\zeta^*$, we need only to find $\lambda$ as the solution to a $\partial$-problem \cite{ma2}:
\[
\lambda_z=(\zeta^*)^{-2}.
\]
Since $(\zeta^*)^{-2}=e^{2w}=2x_z\cdot x_{\zb}$ and $x_{z\zb}=0$, any solution $\lambda$ is the following form:
\[
\lambda=2 (x\cdot x_{\zb}+\bar{g}),~~~\text{where $g$ is any holomorphic function}.
\]
Thus in terms of this auxiliary holomorphic function $g$, a general adjoint transform $[\hat{X}]$ of $x$ under a local coordinate $z$ is given by
\begin{align}
\zeta &=\lambda\zeta^*=2(x\cdot x_{\zb}+\bar{g})e^{-w};\\
\mu &=\frac{-2\zeta_z}{\zeta}
=\frac{-2x_z\cdot x_{\zb}}{x\cdot x_{\zb}+\bar{g}}+2w_z
=\frac{-e^{2w}}{x\cdot x_{\zb}+\bar{g}}+\mu^*;\\
\hat{X}&=N+\bar\mu X_z +\mu X_{\zb}+\frac{1}{2}\abs{\mu}^2 X \notag,\\
e^{-w}\hat{X}^\top&=
\begin{pmatrix}\vec{0} \\ ~1\\ -1\end{pmatrix}
  -\frac{1}{x\!\cdot\! x_z\!+\!g}
\begin{pmatrix}x_z \\~x\!\cdot\! x_z \\-x\!\cdot\! x_z\end{pmatrix}
  -\frac{1}{x\!\cdot\! x_{\zb}\!+\!\bar{g}}
\begin{pmatrix}x_{\zb} \\~x\!\cdot\! x_{\zb} \\-x\!\cdot\! x_{\zb}\end{pmatrix}
  +\frac{x_z\!\cdot\! x_{\zb}}{\abs{x\!\cdot\! x_z\!+\!g}^2}
\begin{pmatrix}
x \\ \frac{-1+x\cdot x}{2}\\ \frac{-1-x\cdot x}{2}\end{pmatrix}.
\label{eq-xhat0}
\end{align}
In the final step we used \eqref{eq-x},\eqref{eq-x1}, \eqref{eq-x6} and $\mu^*=2w_z$. Then it is crucial to notice
\[
\hat{X}
=e^w\frac{x_z\!\cdot\! x_{\zb}}{\abs{x\cdot x_z+g}^2}
\left(\hat{x},\frac{-1+\hat{x}\cdot\hat{x}}{2},
\frac{-1-\hat{x}\cdot\hat{x}}{2}\right)
\]
with
\begin{equation}\label{eq-xhat1}
   \hat{x}=x
   -\frac{x\cdot x_{\zb}+\bar{g}}{x_z\cdot x_{\zb}}x_z
   -\frac{x\cdot x_z+g}{x_z\cdot x_{\zb}}x_{\zb}.
\end{equation}
One immediately recognizes that after taking stereographic projection back to the original $\R^n\supset x(M^2)$, the adjoint transform $[\hat{X}]$ is represented by $\hat{x}$ in the same affine space.

In summary, we have obtained
\begin{theorem}
Any adjoint transform $\hat{x}$ of a super-conformal minimal surface $x:M^2\to\R^n$ is given in \eqref{eq-xhat1} for some meromorphic $g$. Conversely, any $\hat{x}$ given in \eqref{eq-xhat1} for some meromorphic $g$ is also an adjoint surface of $x$.
\end{theorem}

\begin{remark}
Note that in \eqref{eq-xhat1}, $g$ is allowed to have poles. In particular, the dual surface can be recovered by taking $g\equiv \infty$ (which means that this holomorphic mapping to $\mathbb{C}P^1$ degenerates to one point). From the global viewpoint, in general we should regard this $g$ in \eqref{eq-xhat1} as a meromorphic 1-form on $M^2$. Conversely,
from a super-conformal Euclidean minimal surface $x$ and its adjoint surface $\hat{x}$ we can find out this 1-form given by
\begin{equation}
g dz=(x-\hat{x})\cdot x_z dz.
\end{equation}
\end{remark}

\begin{remark}
One can check that the surface given by \eqref{eq-xhat1} agrees with the original geometric characterization of an adjoint transform (Section~3.2). First, at any point $p\in M^2$, \eqref{eq-xhat1} implies that $\hat{x}(p)$ is contained in the tangent plane $T_p x(M^2)$ which is exactly the mean curvature sphere of the minimal surface $x(M^2)$ at $p$. Next, differentiating \eqref{eq-xhat1} and simplifying the result, we get
\begin{equation}\label{eq-xhatz1}
\hat{x}_z
=-\frac{x_{\zb}\cdot x+\bar{g}}{x_z\cdot x_{\zb}}
\left(x_{zz}-\frac{x_{zz}\cdot x_{\zb}}{x_z\cdot x_{\zb}}x_z\right)
+(\cdots)x_{\zb}.
\end{equation}
This implies $\<\hat{x}_z,\hat{x}_z\>=0$; hence $\hat{x}$ is also a conformal map from $M^2$. Finally, by \eqref{eq-xhatz1} we know $\<\hat{x}_z,x_{\zb}\>=0$, which verifies the co-touching property (see Definition~3.1 in \cite{ma2}). So $\hat{x}$ and $x$ satisfy the characterization of a pair of adjoint surfaces.
\end{remark}

\section{The pedal surface and branch points}
One special adjoint transform is given by taking $g=0$ identically in \eqref{eq-xhat1}, i.e.,
\begin{equation}\label{eq-xhat2}
   \hat{x}=x
   -\frac{x\cdot x_{\zb}}{x_z\cdot x_{\zb}}x_z
   -\frac{x\cdot x_z}{x_z\cdot x_{\zb}}x_{\zb}.
\end{equation}
This is exactly the classical construction of \emph{pedal surface}, i.e., for any $p\in M^2$, $\hat{x}(p)$ is exactly the foot of perpendicular from the origin $\vec{0}$ to the tangent plane $T_p x(M^2)$.

More generally, we can take any fixed point $x_0\in \R^n$ and consider the holomorphic function $g=-2x_0\cdot x_z$ (which can be viewed as a combination of coordinate functions of the ambient space).
Then the corresponding $\hat{x}$ is the pedal surface of $x$ with respect to this fixed $x_0$ (called the \emph{pedal point}). In summary, we have proved
\begin{theorem}
The pedal surfaces $\hat{x}$ of a super-conformal Euclidean minimal surface $x:M^2\to\R^n$ are a family of adjoint transforms of $x$, depending on the choice of the pedal point $x_0$, i.e., $n$ real parameters.
\end{theorem}

 For the purpose of constructing immersed examples, it is important to answer the following question:  \\

{\em For a pedal surface given in \eqref{eq-xhat2}, when will it be immersed without any  branch point?}\\

 \eqref{eq-Q1} and \eqref{eq-xhatz1} tell us the information of $\hat{x}_z$, which takes a simple form as below when $g=0$:
\begin{equation}\label{eq-xhatz2}
\hat{x}_z
=-\frac{x_{\zb}\cdot x}{x_z\cdot x_{\zb}}Q
-\frac{Q\cdot x}{x_z\cdot x_{\zb}}x_{\zb},~~~
\text{with}~Q=x_{zz}-\frac{x_{zz}\cdot x_{\zb}}{x_z\cdot x_{\zb}}x_z~\text{being the Hopf differential}.
\end{equation}
As a consequence, when $Q=0$, $\hat{x}$ has a branch point.

In the complement of those umbilic points, since $x_z\ne 0$ by the assumption that $x$ is an immersion, $\hat{x}_z=0$ if, and only if, $x_{zz}\cdot x=x_z\cdot x=0$ at one point.
In general, for a pedal surface constructed using another fixed point $x_0$, it has a branch point if, and only if,
\begin{equation}\label{eq-branch}
x_{zz}\cdot (x-x_0)=x_z\cdot (x-x_0)=0,
\end{equation}
at some point.

Now we assert that one can always remove such branch points by a re-choice of the origin (i.e., the fixed point $x_0$ appearing in the construction of the pedal surface).
At any point $p\in M^2$, by the assumption of immersion and no umbilic points, the real and imaginary parts of $\{x_{zz}(p),x_z(p)\}$ span a 4-dimensional real subspace of $\R^n$.
Thus the solutions $x_0$ to the systems \eqref{eq-branch} form a $(n-4)$-dimensional affine subspace.
When $p$ is taken all over $M^2$, such points will form a subset
\[
\digamma=\bigcup_{p\in M}\{\mathbf{v}\in\R^n|x(p)-\mathbf{v}\perp x_z(p),x_{zz}(p)\}
\]
of $\R^n$ with dimension no greater than $n-2$. Thus if we choose $x_0$ in $\R^n\setminus \digamma$, which is an open dense subset of $\R^n$, the corresponding adjoint surface will have no branch points on $M^2$ (except those umbilic points of $x(M)$). In summary we have proved
\begin{proposition}
A generic adjoint surface $\hat{x}$ of a super-conformal minimal surface $x$ immersed in $\R^n$ has no branch points (except at the umbilic points and the ends of the first surface $x$).
\end{proposition}
When the codimension is bigger than $1$, generally it is easy to remove the known umbilic points by some deformations. Since
our aim is to produce immersed examples of Willmore 2-spheres using the construction of pedal surfaces as above, we turn to the final possible source of branch points: the \emph{ends} of $x$.

Assume that $M^2=D^2\setminus \{0\}$ with $z=0$ being the end of $x$.  Although $x(M^2)$  may not be able to be extended smoothly to this end when $x$ is viewed as a surface in $S^n$,  for the pedal surface $\hat{x}$ this still stays possible.
To control the behavior of $\hat{x}$ at the end of $x$ so that $\hat{x}$ has not only a smooth limit, but also is immersed when viewed as surfaces in $S^n$, we need to analyze the Laurent expansion of $x$ and $\hat{x}$ at $z=0$.

\begin{lemma}\label{lem-min-ends}
Let $x:D^2\setminus \{0\}\to \R^n$ be a super-conformal algebraic minimal surface and a full immersion of the punctured disk in $\R^n$ with $n\geq 5$. Suppose $x_z dz$ has no residue at the end $z=0$. Set
\begin{equation}\label{eq-laurent1}
x=2Re\left(\vec{v}_{-m}z^{-m}+\vec{v}_{k-m}z^{k-m}
+\sum_{j>k-m}^{+\infty}\vec{v}_jz^j\right),
\end{equation}
where $m,k$ are positive integers, and the coefficient vectors $\vec{v}_{-m}, \vec{v}_{k-m}$ are assumed to be $\C$-linear independent.
Then we have
\begin{equation}\label{eq-laurent2}
\hat{x}=2Re\left(\frac{k}{m}\Big(\vec{v}_{k-m}
-\frac{\bar{\vec{v}}_{-m}\cdot \vec{v}_{k-m}}
{|\vec{v}_{-m}|^2}\vec{v}_{-m}\Big)z^{k-m}+o(|z|^{k-m})\right)
\end{equation}
for the pedal surface $\hat{x}$ defined in \eqref{eq-xhat2}.

As a corollary, when $k-m<0$, $\hat{x}$ is immersed at $z=0$ (after an inversion in $\R^n$) if only if
$k-m=-1$. When $k-m>0$, $\hat{x}$ is immersed at $z=0$ if and only if $k-m=1$. (In the case $k-m=0$, the effect of adding the vector $\vec{v}_0$ is adding a constant plus some quantity in the order $o(\frac{1}{|z|})$, which can be ignored; then it is reduced to the case $k-m>0$.)
\end{lemma}
\begin{proof}
By the conformal property $x_z\cdot x_z\equiv 0$ there should be
$\vec{v}_{-m}\cdot \vec{v}_{-m}=0=\vec{v}_{-m}\cdot \vec{v}_{k-m}$. The rest of this proof is just power series expansion (with respect to the variables $z$ and $\zb$), where one needs only to keep track of the lowest and the second lowest order terms. We omit the details of this straightforward computation.
\end{proof}
Notice that the case $k-m=-1$ is similar to the \emph{flat ends} of a minimal surface discussed by Bryant (on page 47-48 of \cite{bryant1}). It is easy to show that around such a flat end, $\hat{x}$ extends smoothly to an immersed regular surface in $S^n$. This can be stated as a more general result as below whose proof is straightforward. (When $k-m<-1$, after inversion we will get a branch point at this end.)
\begin{proposition}\label{prop-flatend} ( page 47-48 of \cite{bryant1})
Let $\hat{x}:D^2\setminus\{0\}\to \R^n$ be one end of an immersed real analytic surface. Assume that at $z=0$ it has expansion $\hat{x}=2Re(\frac{1}{z})+h(z,\zb)$ where $h(z,\zb)$ denotes a convergent power series $(z,\zb)$ in a small neighborhood of $z=0$. Then $I\circ \hat{x}$ extends smoothly to $z=0$ where $I$ is an inversion with respect to an arbitrary hyper-sphere in $\R^n$.
\end{proposition}

\begin{remark}
We can always assume that $x$ has the expansion given by \eqref{eq-laurent1} with $\vec{v}_{-m}\nparallel \vec{v}_{k-m}$. Otherwise, suppose the Laurent series is $\vec{v}_{-m}(z^{-m}+a_1 z^{1-m}+\cdots+a_{k-1}z^{k-1-m})+\vec{v}_{k-m}z^{k-m}+\cdots$ for some integer $k,m$ and $\vec{v}_{-m}\nparallel \vec{v}_{k-m}$. We choose a new holomorphic coordinate $\tilde{z}$ suitably so that $\tilde{z}^{-m}=z^{-m}+a_1 z^{1-m}+\cdots+a_{k-1}z^{k-1-m}$. The local existence of such a $\tilde{z}$ is easy to prove by a standard argument. Then the Laurent expansion of $x(\tilde{z})$ has the desired form.
\end{remark}

We summarize the above conclusions as below. This will be used later in the construction of Willmore 2-spheres in $S^5$.

\begin{theorem}\label{prop-immersion}
Let $x:M^2=\overline{M}\setminus\{p_1,\cdots,p_k\}\to \R^n$ be a complete minimal surface defined on a compact Riemann surface $\overline{M}$ with ends $\{p_1,\cdots,p_k\}$. Suppose:

(i1) $x$ is immersed;

(i2) $x$ has no umbilic points, i.e. $x_{zz}\nparallel x_z$ for any local complex coordinate $z$;

(i3) $x_{zz}\cdot x$ and $x_z\cdot x$ never vanish simultaneously at one point;

(i4) At each end $p_j$, if we take a coordinate with $z(p_j)=0$, then $x$ has the following expansion
\[
x=2Re\left(\vec{v}_{-m}\frac{1}{z^m}+\vec{v}\frac{1}{z}
+O(1)\right),~~\text{or}~~
x=2Re\left(\vec{v}_{-m}\frac{1}{z^m}+\vec{v}_0+\vec{v} z
+O(|z|^2)\right)
\]
where $m\ge 2$, and $\vec{v}_{-m},\vec{v}\in \C^n$ are linearly independent over $\C$ (thus both are non-zero).

Then the pedal surface $\hat{x}$ given by \eqref{eq-xhat2} extends to the whose $\overline{M}^2$ as a closed Willmore surface conformally immersed in $S^n$. In particular, the condition (i3) can always be achieved by a re-choice of the pedal point (the origin).
\end{theorem}

\begin{remark}\label{rem-end-gaussmap}
Lemma~\ref{lem-min-ends} provides a negative answer to a problem mentioned in the end of the introduction: \vspace{2mm}

\emph{If the original Willmore surface is smooth and analytic, can we expect the adjoint transform being good enough so that the conformal Gauss map still extends smoothly to those possible singularities?}\vspace{2mm}

A counterexample of such a pair of adjoint Willmore surfaces can be constructed as below using Lemma~\ref{lem-min-ends}. Let $x$ be a super-conformal Euclidean minimal surface, and $\hat{x}$ its adjoint transform; $p$ is taken to be the point $z=0$ and assume $m\ge 2, k=m-1$. Set
\begin{equation}\label{eq-flatend}
x=2Re\left(\vec{v}_{-m}\frac{1}{z^m}+\vec{v}_{-1}\frac{1}{z}
+O(1)\right),~~
\hat{x}=2Re\left((\cdots)\frac{1}{z}+O(1)\right).
\end{equation}
In this example, $\hat{x}$ has a \emph{flat end} in $\R^n$ which can be compactified smoothly in $S^n$ with a smooth conformal Gauss map around $z=0$ by Proposition~\ref{prop-flatend}.
On the contrary, for the minimal surface $x$ (which is also an adjoint transform of $\hat{x}$), its mean curvature spheres, i.e., those tangent planes, do not have a limit when $z\to 0$, since $x$ does not have an asymptotic 2-plane at this end.

This phenomena shows an interesting difference from the codim-2 case. For a Willmore surface in $S^4$, the adjoint transform is essentially equivalent to the \emph{2-step B\"acklund transform} (which is just a suitable composition of two 1-step B\"acklund transforms) introduced in \cite{BFLPP}. It has been showed in \cite{Leschke} that the mean curvature sphere congruences of 1-step B\"acklund transforms always extend smoothly across the possible branch points.
\end{remark}

\section{Examples of immersed Willmore 2-spheres in $S^5$ and $S^6$ which are not S-Willmore}

In this section, we will derive new examples of Willmore two spheres in $S^5$ and $S^6$ by constructing pedal surfaces from suitable minimal surfaces in $\R^5$ or $\R^6$.

The first subsection shows the concrete construction and explicit expressions of a minimal surface $x$ in $\R^6$ together with its adjoint Willmore surface $\hat{x}$, where $\hat{x}$ is an immersion in $S^6$ and it is not S-Willmore. Both $x$ and $\hat{x}$ are totally isotropic in $S^6$. $x$ has a branched point at infinity.

In subsection 8.2, we begin from a superconformal minimal surface $x$ in $\R^n$, which is not totally isotropic and has 3 ends. For its pedal surfaces to be immersed, one needs $n\geq 6$. This way, we obtain a Willmore two-sphere full in $S^6$, which is not S-Willmore and has non-isotropic Hopf differential.

In subsection 8.3, we finally obtain a Willmore two-sphere $\hat{x}$ in $S^5$, which is not S-Willmore and has non-isotropic Hopf differential. It is a pedal surfce of a superconformal minimal surface $x$ in $\R^5$, where $x$  has 4 ends.

%We refer to \cite{bryant1}, \cite{Bryant1988}, \cite{kusner} and \cite{PX} for some useful technics in the construction of suitable minimal surfaces in $\R^n$.

\subsection{Example 1}
Below we describe a minimal surface $x$ in $\R^6$, together with one adjoint surface $\hat{x}$, defined on $\C$ using the coordinate $z$, which is first derived in \cite{Wang}:
\begin{equation}\label{eq-example1}
x=\begin{pmatrix}
\frac{i}{4}\left(z-\zb\right)\\
-\frac{1}{4}\left(z+\zb\right)\\
-\frac{i}{2}\left(\frac{1}{\zb}-\frac{1}{z}\right)\\
\frac{1}{2}\left(\frac{1}{\zb}+\frac{1}{z}\right)\\
\frac{i}{6}\left(z^2-\zb^2\right)\\
-\frac{1}{6}\left(z^2+\zb^2\right)
\end{pmatrix},~~
\hat{x}=\frac{1}{1+\frac{r^4}{4}+\frac{4r^6}{9}}
\begin{pmatrix}
\left(1+\frac{r^6}{9}\right)\frac{i}{2}\left(z-\zb\right)\\
\left(1+\frac{r^6}{9}\right)\frac{-1}{2}\left(z+\zb\right)\\
\left(\frac{r^2}{4}+\frac{r^4}{3}\right)i\left(\zb-z\right)\\
\left(\frac{r^2}{4}+\frac{r^4}{3}\right)\left(\zb+z\right)\\
\left(1-\frac{r^4}{12}\right)\frac{i}{2}\left(z^2-\zb^2\right)\\
\left(1-\frac{r^4}{12}\right)\frac{-1}{2}\left(z^2+\zb^2\right)
\end{pmatrix}.
\end{equation}
Obviously $x$ is a totally isotropic conformal minimal surface.
Denote $r^2=|z|^2$. We compute
\begin{equation}\label{eq-example11}
\begin{aligned}
x\cdot x=\frac{1}{r^2}\Big(1+\frac{r^4}{4}+\frac{r^6}{9}\Big),
~\Rightarrow~~~&
x_z\cdot x=\frac{(x\cdot x)_z}{2}=\frac{-1}{2zr^2}
 \Big(1-\frac{r^4}{4}-\frac{2r^6}{9}\Big),\\
\Rightarrow~~~& x_{zz}\cdot x=(x_z\cdot x)_z=\zb^2 \Big(\frac{1}{9}+\frac{1}{r^6}\Big)\ne 0,\\
& x_z\cdot x_{\zb}=(x_z\cdot x)_{\zb}=
\frac{1}{2r^4} \Big(1+\frac{r^4}{4}+\frac{4r^6}{9}\Big)>0.
\end{aligned}
\end{equation}
Then it is easy to verify that $\hat{x}$ is the pedal surface of $x$ with pedal point $\vec{0}\in \R^6$ using \eqref{eq-xhat2}. In particular, $\hat{x}$ is also totally isotropic.

The  Willmore two-sphere $\hat{x}$ first appeared in Section 5.3 of \cite{DoWa1} by J. Dorfmeister and P. Wang, as the first example of Willmore two-sphere in $S^6$ which is not S-Willmore. Using the celebrated DPW method and a simplest choice of meromorphic potential, they were able to find this $\hat{x}:\C\to \R^6$ as an immersed Willmore surface which extends to $z=\infty$ smoothly in $S^6$. This provided the first example of a Willmore 2-sphere in $S^n$ which is not S-Willmore. Note that in \cite{DoWa1} this example was represented in $S^6$ and denoted by $x_{\lambda}$ with $\lambda=1$, where $\lambda\in S^1$ is the parameter of a loop in the theory of loop groups.

The pair of the adjoint surfaces $x$ and $\hat{x}$ are derived in the same spirit via another harmonic map related with Willmore surfaces, which was discovered  by Helein \cite{Helein} and Xiang Ma \cite{ma2} in different approaches. We refer to \cite{Wang} for details.
\\

\textbf{Claim 1.} ~$\hat{x}$ is an immersion from $\C\cup\{\infty\}$ to $S^6$.

This can be checked directly using \eqref{eq-example1}. For $\hat{x}$ one can find a local lift, written as a column vector:
\begin{equation}\label{eq-example12}
\hat{X}=
\begin{pmatrix}
i\left(z-\zb\right)\left(1+\frac{r^6}{9}\right)\\
-\left(z+\zb\right)\left(1+\frac{r^6}{9}\right)\\
i\left(\zb-z\right)\left(\frac{r^2}{2}+\frac{2r^4}{3}\right)\\
\left(\zb+z\right)\left(\frac{r^2}{2}+\frac{2r^4}{3}\right)\\
i\left(z^2-\zb^2\right)\left(1-\frac{r^4}{12}\right)\\
-\left(z^2+\zb^2\right)\left(1-\frac{r^4}{12}\right)\\
1-r^2-\frac{3r^4}{4}+\frac{4r^6}{9}-\frac{r^8}{36}\\
1+r^2+\frac{5r^4}{4}+\frac{4r^6}{9}+\frac{r^8}{36}
\end{pmatrix}.
\end{equation}
Computation shows $\frac{1}{2}\<\hat{X}_z,\hat{X}_{\zb}\>=
1+4r^2+\frac{r^4}{4}+\frac{2r^6}{9}+\frac{4r^8}{9}+
\frac{r^{10}}{36}+\frac{r^{12}}{81}>0$. Hence
\[
\hat{y}=\frac{1}{1+r^2+\frac{5r^4}{4}+\frac{4r^6}{9}
+\frac{r^8}{36}}\hat{X}
\]
gives a conformal immersion of $\C$ into $S^6(1)$. Take a new coordinate $\tilde{z}=1/z$ at $z=\infty$. The induced metric of $\hat{y}$ is
\[
\<\hat{y}_z,\hat{y}_{\zb}\>
\cong \frac{36^2}{r^{16}}\<\hat{X}_z,\hat{X}_{\zb}\>
\cong 32\frac{|dz|^2}{r^4}=32|d\tilde{z}|^2.
\]
So $y$ extends to be an immersion $\C\cup\{\infty\}\to S^6$.

Alternatively, we can verify this claim using Theorem~\ref{prop-immersion}. The condition (i1)-(i3)
hold true on $\C\setminus\{0\}$ by \eqref{eq-example11}.
As to condition (i4), $x$ has the desired form at both ends. Compared to Lemma~\ref{lem-min-ends}, here $m=-1,k-m=1$ at $z=0$, $m=-2,k-m=-1$ at $z=\infty$. Then our claim follows.\\

\textbf{Claim 2.} $\hat{x}$ is not S-Willmore.

In \cite{DoWa1}, this was part of Theorem~5.12, which follows from Theorem~3.10 and Corollary~3.13 in that paper, and depended on the fact that the corresponding normalized potential $\hat{B}_1$ has rank $2$.

Here we need only to check that the subspace spanned by frame vectors \[\{\hat{X},\hat{X}_z,\hat{X}_{\zb},\hat{X}_{z\zb},
\hat{X}_{zz},\hat{X}_{zz\zb}\}\] has dimension $6$, i.e., the corresponding matrix is of full rank, when $z=0$. Using \eqref{eq-example12} which is itself a vector-valued polynomial, we check that this matrix at $z=0$ is given by
\[
(\hat{X},\hat{X}_z,\hat{X}_{\zb},\hat{X}_{z\zb},
\hat{X}_{zz},\hat{X}_{zz\zb})_{z=0}=\begin{pmatrix}
0 & i & -i & 0 & 0 & 0\\
0 & -1 & -1 & 0 & 0 & 0\\
0 & 0 & 0 & 0 & 0 & -i\\
0 & 0 & 0 & 0 & 0 & 1\\
0 & 0 & 0 & 0 & 2i & 0\\
0 & 0 & 0 & 0 & -2 & 0\\
1 & 0 & 0 & -1 & 0 & 0\\
1 & 0 & 0 & 1 & 0 & 0\\
\end{pmatrix}.
\]
The rank is $6$. Thus the Hopf differential of $\hat{x}$, denoted by $\hat\kappa$, is linearly independent to $\hat{D}_{\zb}\hat\kappa$, at least in a neighborhood of $z=0$.
By the real analytical property of Willmore surfaces, $\hat{x}$ is not S-Willmore in an open dense subset. Claim~2 is proved.

\subsection{Example 2}

To find more examples of Willmore 2-spheres which are not S-Willmore, we have reduced the problem to constructing  super-conformal minimal surfaces in $\R^n$ using rational functions on $\C$. The general procedure is prescribing the number of ends and the end behavior, then solving the coefficient vectors so that $x$ satisfies
\begin{equation}\label{eq-mini-isotropic}
x_z\cdot x_z=x_{zz}\cdot x_{zz}=0.
\end{equation}
This amounts to solving an algebraic equation system.
When the number of ends is small, usually we find only trivial solutions, corresponding to totally isotropic examples which exist only in even-dimensional space.

Here we try to construct a super-conformal minimal surface $x$ in $\R^n$ with genus $0$ and three ends, satisfying the formula \eqref{eq-flatend}. It is always possible to assign these three ends at $z=0,1,\infty$ (up to a suitable fraction linear transformation on $\C$). The simplest candidate surfaces are of the form
\begin{equation}\label{eq-example2}
x_z=\frac{\vec{u}_3}{z^3}+\frac{\vec{u}_2}{z^2}
+\frac{\vec{v}_3}{(z-1)^3}+\frac{\vec{v}_2}{(z-1)^2}
+\vec{w}_2+\vec{w}_3 z.
\end{equation}
For these surfaces, we have the following result.

\begin{proposition}\label{prop-mini-s6}
There exists a conformal minimal surface $x$ in $\R^n$ defined by \eqref{eq-example2} on the complex plane, which is supposed to be super-conformal (1-isotropic) but not $2$-isotropic, i.e., $x_{zz}\cdot x_{zz}\equiv 0, x_{zzz}\cdot x_{zzz}\not\equiv 0$. Moreover, we have the following conclusions.

(1) All such examples lie in a unique associated family of minimal surfaces in an affine subspace $\R^6$ (up to rigid motions and dilations).

(2) It has an adjoint surface $\hat{x}$ which extends to be an immersed Willmore 2-sphere in $S^6$. $\hat{x}$ is not super-conformal.
\end{proposition}
\begin{remark}
Although this example can not be contained in $\R^5$ as we originally expected, compared with Example~1, This is the new type of examples (NOT totally isotropic) as predicted in \cite{DoWa1} (see their discussions in Section 5.3).
\end{remark}
\begin{proof}
It follows from \eqref{eq-example2} that
\begin{equation}\label{eq-example21}
x_{zz}=\frac{-3\vec{u}_3}{z^4}+\frac{-2\vec{u}_2}{z^3}
+\frac{-3\vec{v}_3}{(z-1)^4}+\frac{-2\vec{v}_2}{(z-1)^3}+
\vec{w}_3.
\end{equation}
As usual, the conformal condition $x_z\cdot x_z\equiv 0$ immediately implies
\begin{equation}\label{eq-example22}
0=\vec{u}_j\vec{u}_k=\vec{v}_j\vec{v}_k=\vec{w}_j\vec{w}_k,~~
1\le j,k\le 2.
\end{equation}
Note that $\vec{u}_j\vec{u}_k$ stands for the $\C$-linear extension of the Euclidean inner product between $\vec{u}_j$ and $\vec{u}_k$; the dot $\cdot$ is omitted.
The 1-isotropic condition $0\equiv x_{zz}\cdot x_{zz}=\frac{f(z)}{z^4(z-1)^4}$ then implies
\begin{multline*}
0=f(z)=9\vec{u}_3\vec{v}_3
+6\vec{u}_3\vec{v}_2 (z-1)
+6\vec{u}_2\vec{v}_3 z
+4\vec{u}_2\vec{v}_2 z(z-1)\\
-3\vec{u}_3\vec{w}_3 (z-1)^4
-2\vec{u}_2\vec{w}_3 z(z-1)^4
-3\vec{v}_3\vec{w}_3 z^4
-2\vec{v}_2\vec{w}_3 z^4(z-1).
\end{multline*}
The coefficients must all vanish. Solving this linear equation system is easy: depending on two arbitrary complex parameters $a,b$, the solutions are
\begin{equation}\label{eq-example23}
\begin{aligned}
\vec{u}_2\vec{w}_3&=\vec{u}_3\vec{w}_3=\vec{v}_3\vec{w}_3
=-\vec{v}_2\vec{w}_3=a,\\
\vec{u}_2\vec{v}_2&=\frac{5}{2}a,~~
\vec{u}_3\vec{v}_2=-\vec{u}_2\vec{v}_3=b,~~
\vec{u}_3\vec{v}_3=\frac{2}{3}b+\frac{1}{3}a.
\end{aligned}
\end{equation}
Inserting \eqref{eq-example2} into $x_z\cdot x_z\equiv 0$ and using \eqref{eq-example22}, \eqref{eq-example23}, similarly we find
\begin{equation}\label{eq-example24}
b=-2a, \ \hbox{ and } ~\vec{u}_2\vec{w}_2=-\frac{a}{2},~~
\vec{u}_3\vec{w}_2=a,~~
\vec{v}_3\vec{w}_2=-2a,~~
\vec{v}_2\vec{w}_2=\frac{a}{2}.
\end{equation}
Denote $(\vec{e}_1,\cdots,\vec{e}_6)=
(\vec{u}_3,\vec{u}_2,\vec{v}_3,\vec{v}_2,\vec{w}_3,\vec{w}_2)$.
Then the inner product matrix is given by
\begin{equation}\label{eq-example26}
A^{6\times 6}=(\vec{e}_i\vec{e}_j)=
a\cdot\begin{pmatrix}
0 & 0 & -1 & -2 & 1 & 1\\
0 & 0 & 2 & 5/2 & 1 & -1/2\\
-1 & 2 & 0 & 0 &  1 & -2\\
-2 & 5/2 & 0 & 0 & -1 & 1/2\\
1 & 1 & 1 & -1 & 0 & 0\\
1 & -1/2 & -2 & 1/2 & 0 & 0
\end{pmatrix}.
\end{equation}
Since we have assumed that $x$ is not totally isotropic, $a\ne 0$. Without loss of generality, we may assume $a=1$. (Any other solution $x_a$ is one member in the associated family of this minimal surface $x$ up to a dilation.)

To find out $x$ we have to find a $6\times n$ matrix $B^{6\times n}$ consisting of the row vectors $\{e_i\}$ so that
\[ BB^\top=A \]
where $A^{6\times 6}$ is given as above. It is easy to verify $\mathrm{rank}(A)=6$. Because $A$ is a non-singular symmetric real matrix, there exists a \emph{real} and \emph{non-singular} matrix $B$ such that $B\cdot B^\top=A$ if and only if $n\ge 6$.
Suppose this is the case, then corresponding row vectors $\{
\vec{u}_3,\vec{u}_2,\vec{v}_3,\vec{v}_2,\vec{w}_3,\vec{w}_2\}$ span a 6-dimensional complex subspace $\C^6$. Any other solution $B'$ differs from the previous $B$ by $B'=BP$ where $P\in O(6)$ is an orthogonal matrix.

Moreover, since $A$ is real symmetric non-singular matrix, the eigenvalues of $A$ are all non-zero real numbers (indeed the signature is $(3,3)$), and all eigenvectors are real. This guarantees that the real and imaginary parts of $\{
\vec{u}_3,\vec{u}_2,\vec{v}_3,\vec{v}_2,\vec{w}_3,\vec{w}_2\}$
span a 6-dimensional real subspace $\R^6$.
Inserting them into \eqref{eq-example2} and taking integration, we know $x$ is located in a 6-dimensional affine subspace.
This finishes the proof to the first conclusion.

The second conclusion follows immediately by taking a pedal surface with a suitable chosen pedal point and using Theorem~\ref{prop-immersion}. By construction this is a 1-isotropic but not totally isotropic surface in $\R^6$. We can verify the conditions of Theorem~\ref{prop-immersion} for $x$ one by one:

(i1) It is an immersion since $x_z$ never vanishes at any point by \eqref{eq-example2} and the linear independence of $\{
\vec{u}_3,\vec{u}_2,\vec{v}_3,\vec{v}_2,\vec{w}_3,\vec{w}_2\}$.

(i2) $x_{zz}$ is always linearly independent to $x_z$ by using \eqref{eq-example2} and \eqref{eq-example21}.

(i3) This condition is satisfied for generic choice of $x_0$.

(i4) Its end behavior is as desired.\\
This completes the proof.
\end{proof}

\subsection{Example 3}

Using the same idea as the previous section, we construct a 1-isotropic minimal surface defined on
\[
\mathbb{C}\cup\{\infty\}\setminus \{0,\epsilon_1,\epsilon_2,\epsilon_3\}~~\text{where}~
\epsilon_k=e^{2k\pi i/3},~k=1,2,3.
\]
with prescribed end behavior as below:
\begin{align}
x_z&=
\frac{\vec{a}_3}{(z-\epsilon_1)^3}
+\frac{\vec{a}_2}{(z-\epsilon_1)^2}
+\frac{\vec{b}_3}{(z-\epsilon_2)^3}
+\frac{\vec{b}_2}{(z-\epsilon_2)^2}
+\frac{\vec{c}_3}{(z-1)^3}
+\frac{\vec{c}_2}{(z-1)^2}
+\frac{\vec{r}_3}{z^3}
+\frac{\vec{r}_2}{z^2}  \notag \\
&=\frac{\Phi}{(z^3-1)^3z^3}. \label{eq-example3}
\end{align}
Here $\Phi=\Phi(z)=\sum_{j=0}^{10} \vec{v}_j z^j$ should be a vector-valued polynomial of degree no more than $10$.
The vanishing of all residues is equivalent to
\begin{equation}\label{eq-example31}
\vec{v}_6 = 20\vec{v}_0 + 5\vec{v}_3 + 2\vec{v}_9,~~
\vec{v}_7 = 14\vec{v}_1 + 2\vec{v}_4 + 5\vec{v}_{10}.
\end{equation}
Together with the isotropic conditions $x_z\cdot x_z=0=x_{zz}\cdot x_{zz}$, we obtain a system of linear equations on the coefficients of the inner products $\lambda_{jk}=\vec{v}_j\vec{v}_k$, which can be solved using Maple. The general solutions depend on 6 parameters. If we put the Ansatz
\begin{equation}\label{eq-example32}
\lambda_{j,10}=0, ~~\forall~j;~~~\lambda_{08}=1,
\end{equation}
then almost all coefficients vanish except
\begin{equation}\label{eq-example33}
\lambda_{08}=  1,
\lambda_{35}=-16,
\lambda_{38}=-20,
\lambda_{44}= 30,
\lambda_{59}= 20.
\end{equation}
Substitute these into $(\lambda_{jk})^{11\times 11}$ and use \eqref{eq-example31}. The result is a matrix of rank $5$, which can be realized as inner products $\lambda_{jk}=\vec{v}_j\vec{v}_k$ for vectors $\vec{v}_j$ in a 5-dimensional space.
We omit the details and give directly the final result. Set
\[
 E_1=\left(
       \begin{array}{ccccc}
         1 & i & 0 & 0 & 0 \\
       \end{array}
     \right),\  E_2=\left(
       \begin{array}{ccccc}
        0 & 0 & 1 & i &  0 \\
       \end{array}
     \right),\ e_5= \left(
       \begin{array}{ccccc}
         0 & 0 & 0 & 0 & 1 \\
       \end{array}
     \right).
 \]
Let
\begin{equation}\label{eq-example34}\begin{split}
\Phi= (1-20z^3-80z^6)E_1+ \frac{z^8}{2}\bar{E}_1
-(8z^3+20z^6-10z^9)E_2
+z^5\bar{E}_2+\sqrt{30}(z^4+2z^7)e_5.
 \end{split} \end{equation}
with
 \[\begin{split}
\Phi_z= (-60z^2-480z^5)E_1+ 4z^7\bar{E}_1-(24z^2+120z^5-90z^8)E_2+5z^4\bar{E}_2
+\sqrt{30}(4z^3+14z^6)e_5.
 \end{split}\]
From the coefficients of $\bar{E}_1$ and $\bar{E}_2$ we see that $\Phi$ and $\Phi_z$ are linear independent on $\C\setminus\{0\}$. Hence $x$ has no umbilic points. It is also easy to see that
$x$ is immersed. Integration yields
\begin{equation}\label{eq-example-x}
\begin{split}x=x_0+2Re&\left[
\frac{1 +32 z^3}{2 (z^3-1)^2 z^2}E_1
+\frac{1-2 z^3 }{12 (z^3-1)^2 }\bar{E}_1\right.\\
&\hspace{2cm}\left.
+\frac{8z-5z^4}{ (z^3-1)^2}E_2
-\frac{1}{6 (z^3-1)^2}\bar{E}_2
-\frac{\sqrt{30} z^2}{2 (z^3-1)^2}e_5\right].
\end{split}\end{equation}
It is straightforward to compute out the coefficient vectors $\{\vec{a}_3,\vec{a}_2\},~
\{\vec{b}_3,\vec{b}_2\},~
\{\vec{c}_3,\vec{c}_2\}$ appearing in \eqref{eq-example3} and verify that each pair of them
are linearly independent to each other. On the other hand, at the end $z=0$, using \eqref{eq-example3} and \eqref{eq-example34} we see $x_z=\frac{1}{z^3}E_1-20E_1-8E_2+o(z)$.

Thus the conditions of Theorem~\ref{prop-immersion} are satisfied, \emph{except that $x$ has a branch point of order $1$ at $z=\infty$.} Taking a pedal surface $\hat{x}$ with a suitable choice of the pedal point $x_0$, we get a desirable example \emph{if only we can check that $\hat{x}$ is still immersed at $z=\infty$}.

This final step is easy. Taking a new coordinate $w=1/z$, we read from  \eqref{eq-example3} and \eqref{eq-example34} that
\[
x_w=\frac{10w}{(w^3-1)^3}E_2+\frac{w^2}{2(w^3-1)^3}\bar{E}_1+\cdots
=-10w E_2-\frac{w^2}{2}\bar{E}_1+o(w^2).
\]
in a neighborhood of $w=0$. We compute the Hopf differential $Q$ (with respect to this coordinate $w$) using this formula and \eqref{eq-xhatz2} as below:
\[
x_{ww}=-10E_2-w\bar{E}_1+o(w),~~Q=x_{ww}-\frac{x_{ww}\cdot x_{\bar{w}}}{x_w\cdot x_{\bar{w}}}x_w=\frac{-w}{2}\bar{E}_1+o(|w|).
\]
Substituting these into \eqref{eq-xhatz2} we obtain
\[
\hat{x}_w=-\frac{x_{\bar{w}}\cdot x}{x_w\cdot x_{\bar{w}}}Q
-\frac{Q\cdot x}{x_w\cdot x_{\bar{w}}}x_{\bar{w}}
=\frac{-1}{40}\left[\left(\bar{E}_2\cdot x(0)\right)\bar{E}_1+\left(\bar{E}_1\cdot x(0)\right)\bar{E}_2\right]+o(|w|).
\]
Similar to the discussion of Section~7, it is easy to make a choice of the pedal point $x_0$ so that when $w=0$, $x(0)$ is not parallel to $e_5=(0,0,0,0,1)$. Then $\hat{x}$ is also immersed when $w=0,z=1/w=\infty$. We state the conclusion as below.
\begin{proposition}
There exists an adjoint surface $\hat{x}$ of the minimal surface $x$ in \eqref{eq-example-x}, such that by conformally embedding $\R^5$ into $S^5$, $\hat{x}$ becomes a global immersion from $S^2$ into $S^5$. Moreover, $\hat{x}$ is neither super-conformal nor S-Willmore (by Proposition~\ref{isotropic0}).
\end{proposition}
With this example, we see that Case 3 of the classification theorem~\ref{main} does occur.

\appendix

\section{Appendix}
One crucial step in our proof of the classification theorem is that: \textit{A holomorphic sub-bundle spanned by some sections is still well-defined at the possible singularities where these sections are no longer linear independent if some suitable conditions are satisfied.} This depends on the following well-known result.

\begin{lemma}\cite{Chern}\label{lem-chern}
Let $\{\psi_i(z), i=1,\cdots, m\}$ be complex-valued functions which satisfy the differential system
\[
\frac{\partial \psi_i}{\partial \zb}=\sum_j a_{ij}\psi_j,~~~1\le i,j\le m
\]
in a neighborhood of $z=0$, where $\{a_{ij}\}$ are complex-valued $C^1-$functions. Suppose the functions $\{\psi_i(z)\}$ do not vanish identically in a neighborhood of $z=0$. Then\\
\indent (1) The common zeros of $\{\psi_i, i=1,\cdots, m\}$ are isolated;\\
\indent (2) At a common zero of $\{\psi_i\}$, the ratio
$[\psi_1:\cdots:\psi_m]$ tends to a limit.
\end{lemma}

By Chern's lemma we can prove the following result used in the proof of the main theorem.

\begin{lemma}\label{lemma-subbundle}
Given a holomorphic vector bundle $V$ of rank-$m$ over a connected Riemann surface $M$ with
a $\bar\partial$-operator. Let $\psi$ be a section such that $\bar\partial^2 \psi=\lambda\psi$ everywhere for a $C^1$ function $\lambda$.
We also suppose that $\psi$ is non-zero on an open dense subset of $M$. Then the holomorphic sub-bundle $U=Span_{\C}\{\psi,\bar\partial \psi\}$ is well defined on the whole Riemann surface.
\end{lemma}
\begin{proof}
By the assumption, $\psi\wedge \bar\partial\psi$ is a holomorphic section of the associated bundle $\wedge^2 V$. Therefore, either
$\bar\partial\psi, \psi$ are always linearly dependent, or they depend on each other only at several isolated points of $M$.

In the first case, $U=Span_{\C}\{\psi,\bar\partial \psi\}$ is defined with rank-$1$ on the subset where either $\psi$ or $\bar\partial \psi$ is non-zero. The possible exceptional points are the common zeros of $\{\psi,\bar\partial \psi\}$.

In the second case, $U$ is defined with rank-$2$ almost everywhere on $M$. The possible exceptions are those isolated zeros of $\psi\wedge \bar\partial\psi$ (with finite order).

We will show that any possible singularity $p$ is isolated in either case, and $U$ extends continuously to be defined at $p$.

For this purpose, notice that we always take a small neighborhood $\Omega_p$ of $p$ such that $V$ has a local trivialization $\Omega_p\times \C^m$ (which is a holomorphic equivalence).
With respect to this trivialization above, denote
\[
\psi=(\psi_1,\cdots,\psi_m),~~\bar\partial\psi=
(\psi_{m+1},\cdots,\psi_{2m}).
\]
By assumption we have
\[
\frac{\partial \psi_i}{\partial \zb}=\psi_{m+i},
~~\frac{\partial \psi_{m+i}}{\partial \zb}=\lambda \psi_i,~~i=1,\cdots,m.
\]

In the first case, by Lemma~\ref{lem-chern}, the common zeros of
$\psi,\bar\partial\psi$ are isolated, which are the only possible exceptional points. Take such a point $p$ as a common zero of
$\{\psi_1(z),\cdots,\psi_{2m}(z)\}$. By the second conclusion of Lemma~\ref{lem-chern}, the ratio $[\psi_1:\cdots:\psi_m:\psi_{m+1}:\psi_{2m}]$ has a well-defined limit
$[c_1:\cdots:c_m:c_{m+1}:c_{2m}]\in \C P^{2m-1}$. Either $(c_1,\cdots,c_m)$ or $(c_{m+1},\cdots,c_{2m})$ is a non-zero vector in $\C^m$. If both of them are non-zero, then as the limit of parallel vectors, they are also parallel. This provides the desired extension of the line sub-bundle to $p$.

In the second case, let $p$ be an isolated zero (of order $k>0$) of $\psi\wedge \bar\partial\psi$ and suppose the local coordinate is taken so that $z(p)=0$. With respect to a basis $\{v_1,\cdots,v_m\}$ of $\C^m$, we may write
\[
\psi\wedge \bar\partial\psi= z^k\cdot \sum c_{ij} v_i\wedge v_j,~~~(1<i<j<m)
\]
in a neighborhood of $z=0$, where $\{c_{ij}\}$ are holomorphic functions, and at least one of them is non-zero at $z=0$.
These coefficients are exactly the classical Pl\"ucker coordinates
\[
[z^k c_{12}:\cdots:z^k c_{ij}:\cdots]=[c_{12}:\cdots: c_{ij}:\cdots]
\]
of the corresponding subspace $U$ in $\wedge^2\C^m$.
This also gives the Pl\"ucker embedding of the Grassmannian
$Gr(2,\C^m)$ into $\C P^{m(m-1)/2-1}$ as a closed submanifold.
Thus it is obvious that $\psi\wedge \bar\partial\psi$ can be extended continuously to the isolated singularity $p$.
\end{proof}

{\small
\def\refname{Reference}

}

\end{document}